\documentclass[a4paper, reqno]{amsart}

\usepackage[UKenglish]{babel}
\usepackage{amsmath, amssymb, amsthm}
\usepackage{scrextend}
\usepackage[hidelinks]{hyperref}
\usepackage{xpatch}
\usepackage{color}
\usepackage{tikz-cd}
\usepackage{enumitem}
\usepackage{theoremref}

\usepackage[inline,final]{showlabels}
\showlabels{thlabel}

\xpatchcmd{\proof}{\itshape}{\normalfont\bfseries}{}{}
\newtheoremstyle{repeat}{}{}{\itshape}{}{\bfseries}{.}{.5em}{#3, repeated}

\newtheorem{theorem}{Theorem}[section]
\newtheorem{proposition}[theorem]{Proposition}
\newtheorem{lemma}[theorem]{Lemma}
\newtheorem{corollary}[theorem]{Corollary}
\newtheorem{fact}[theorem]{Fact}

\theoremstyle{definition}
\newtheorem{definition}[theorem]{Definition}
\newtheorem{remark}[theorem]{Remark}
\newtheorem{convention}[theorem]{Convention}
\newtheorem{example}[theorem]{Example}

\theoremstyle{repeat}
\newtheorem*{repeated-theorem}{Repeat}

\newcommand{\E}{\mathcal{E}}
\newcommand{\F}{\mathcal{F}}

\newcommand{\B}{\mathcal{B}}
\newcommand{\C}{\mathcal{C}}
\newcommand{\D}{\mathcal{D}}
\renewcommand{\L}{\mathcal{L}}

\newcommand{\g}{{\operatorname{g}}}
\newcommand{\sfo}{{\operatorname{sfo}}}
\newcommand{\fo}{{\operatorname{fo}}}
\renewcommand{\b}{{\operatorname{b}}}

\newcommand{\SStr}[1][\Sigma]{#1 \text{--} \mathbf{Str}} 

\newcommand{\TMod}[1][T]{#1 \text{--} \mathbf{Mod}} 
\newcommand{\Set}{\mathbf{Set}}
\newcommand{\SetFO}{\mathbf{Set}^\fo}
\newcommand{\SetSFO}{\mathbf{Set}^\sfo}
\newcommand{\SetB}{\mathbf{Set}^\b}

\newcommand{\SynG}{\mathbf{Syn}^\g}
\newcommand{\SynFO}{\mathbf{Syn}^\fo}
\newcommand{\SynSFO}{\mathbf{Syn}^\sfo}
\newcommand{\SynB}{\mathbf{Syn}^\b}
\newcommand{\SynConsB}{\mathbf{SynCons}^\b}
\newcommand{\Sh}{\mathbf{Sh}}
\newcommand{\Geom}{\mathbf{Geom}}
\newcommand{\Heyt}{\mathbf{Heyt}}
\newcommand{\SubHeyt}{\mathbf{SubHeyt}}

\newcommand{\FlatCon}{\mathbf{FlatCon}}
\newcommand{\Topos}{\mathbf{Topos}}
\newcommand{\Open}{\mathbf{Open}}
\newcommand{\SubOpen}{\mathbf{SubOpen}}

\renewcommand{\Im}{\operatorname{Im}}
\newcommand{\Sub}{\operatorname{Sub}}

\renewcommand{\phi}{\varphi}
\newcommand{\op}{{\textup{op}}}

\title{Classifying toposes for non-geometric theories}
\date{\today}
\author{Mark Kamsma}
\thanks{The author is supported by the EPSRC grant EP/X018997/1.}
\email[Mark Kamsma]{mark@markkamsma.nl}
\urladdr[Mark Kamsma]{https://markkamsma.nl}
\address[Mark Kamsma]{School of Mathematical Sciences, Queen Mary University of London, London, E1 4NS, UK}
 
\begin{document}

\begin{abstract}
The classifying topos of a geometric theory is a topos such that geometric morphisms into it correspond to models of that theory. We study classifying toposes for different infinitary logics: first-order, sub-first-order (i.e.\ geometric logic plus implication) and classical. For the first two, the corresponding notion of classifying topos is given by restricting the use of geometric morphisms to open and sub-open geometric morphisms respectively. For the last one, we restrict ourselves to Boolean toposes instead. Butz and Johnstone proved that the first-order classifying topos of an infinitary first-order theory exists precisely when that theory does not have too many infinitary first-order formulas, up to intuitionistically provable equivalence. We prove similar statements for the sub-first-order and Boolean case. Along the way we obtain completeness results of infinitary sub-first-order logic and infinitary classical logic with respect to (Boolean) toposes.
\end{abstract}

\maketitle

\tableofcontents

\section{Introduction}
\label{sec:introduction}
A classifying topos for a geometric theory $T$ is a (Grothendieck) topos $\Set[T]$ such that geometric morphisms $f: \E \to \Set[T]$ correspond to models of $T$ internal to $\E$. The correspondence is given by having a generic model $G_T$ in $\Set[T]$ and then the model in $\E$ associated to $f: \E \to \Set[T]$ is $f^*(G_T)$. This unveils the reason for restricting ourselves to geometric logic: that is exactly the fragment of logic that is preserved by the inverse image part $f^*$ of geometric morphisms $f$.

Some geometric morphisms preserve more than just geometric logic, and by restricting the definition of a classifying topos to those geometric morphisms we get a notion of classifying topos for a different fragment of logic. For example, open geometric morphisms are exactly those whose inverse image part preserves full (infinitary) first-order logic ($\L_{\infty, \omega}$ to be precise). Butz and Johnstone \cite{butz_classifying_1998} then defined the \emph{first-order classifying topos} $\SetFO[T]$ for a theory $T$ in infinitary first-order logic to be such that open geometric morphisms $f: \E \to \SetFO[T]$ correspond to models of $T$ in $\E$. Once again, $f: \E \to \SetFO[T]$ corresponds to $f^*(G_T)$, where $G_T$ is the generic model of $T$ in $\SetFO[T]$.

It is well-known that for geometric theories the classifying topos always exists. It turns out that this is not the case for first-order theories. The problem is that there can simply be too many different infinitary first-order formulas, up to provable equivalence in the intuitionistc proof system. By ``too many'' we mean a proper class. We will see in more detail why this is an issue (\thref{classifying-toposes-counterexamples}). The intuition is that the subobjects of the generic model in a first-order classifying topos correspond to equivalence classes of infinitary first-order formulas, and since toposes are well-powered we have that there can only be a set of infinitary first-order formulas up to provable equivalence (modulo the theory $T$). Butz and Johnstone proved that this is in fact the only obstruction to the existence of a first-order classifying topos. So a theory $T$ in infinitary first-order logic has a first-order classifying topos if and only if there is only a set of infinitary first-order formulas up to provable equivalence modulo $T$.

In this paper we build on the work of Butz and Johnstone to prove similar theorems for different fragments of logic: a fragment that we call infinitary sub-first-order logic and infinitary classical first-order logic. Infinitary sub-first-order logic is obtained by adding implication to geometric logic, and this is exactly the fragment that is preserved by the inverse image parts of sub-open geometric morphisms. The notion of sub-first-order classifying topos is then obtained by restricting to sub-open geometric morphisms. The notion of a Boolean classifying topos is obtained by restricting the toposes to Boolean ones, which automatically makes all the geometric morphisms involved open. The existence of these kinds of classifying topos again depends on not having too many formulas, in the relevant logic and modulo the relevant deduction-system.

All these different classifying toposes can be constructed in analogous ways. First we construct a syntactic category, roughly whose objects will be formulas in the relevant logic and arrows will be formulas in the relevant logic that provably (in the relevant deduction-system, modulo the theory) encode the graph of a function. We wish to consider the topos of sheaves on such a syntactic category, for which it needs to be small. To achieve this, we restrict ourselves to formulas in $\L_{\kappa, \omega}$ for some $\kappa$. If there is only a set of formulas, up to provable equivalence, then we can take this $\kappa$ big enough such that the relevant kind of classifying topos is the topos of sheaves on the syntactic category. Even if there are too many formulas, we still have use for this construction in proving completeness theorems for the various fragments of logic.

The proofs and construction for the sub-first-order classifying topos closely follow the strategy of Butz and Johnstone. The strategy for the Boolean classifying topos is different, even though the result looks very similar. This is mainly because we do not restrict the type of geometric morphisms in this case, but the type of topos that we consider. So the main challenge here is to prove that the topos of sheaves on the relevant syntactic category is Boolean.

\textbf{Convention.} We only consider Grothendieck toposes in this paper. So whenever we say ``topos'', a Grothendieck topos is meant.

\textbf{Main results.} We leave the precise statements of the main results to the text, where the appropriate definitions have been made. Here we sum up the essence of the main results.
\begin{itemize}
\item \thref{intuitionistic-completeness-theorems} (intuitionistic completeness): let $T$ be an infinitary sub-first-order theory, then if an infinitary sub-first-order sequent is valid in every model of $T$ in every topos then it is deducible from $T$ in the infinitary sub-first-order deduction-system. A similar statement is true for infinitary first-order logic and geometric logic (which are included in \thref{intuitionistic-completeness-theorems} for completeness' sake), but those are not new in this paper.
\item \thref{classical-completeness-theorem} (classical completeness): let $T$ be an infinitary first-order theory, then if an infinitary first-order sequent is valid in every model of $T$ in every Boolean topos then it is deducible from $T$ in the infinitary classical deduction-system.
\item \thref{sub-first-order-classifying-topos-existence} (existence of sub-first-order classifying toposes): let $T$ be an infinitary sub-first-order theory, then the sub-first-order classifying topos $\SetSFO[T]$ exists iff there is only a set of infinitary sub-first-order formulas, up to provable equivalence in the infinitary sub-first-order deduction-system modulo $T$.
\item \thref{topos-as-sub-first-order-classifying-topos}: every topos is the sub-first-order classifying topos of some infinitary sub-first-order theory $T$.
\item \thref{boolean-classifying-topos-existence} (existence of Boolean classifying toposes): let $T$ be an infinitary first-order theory, then the first-order classifying topos $\SetB[T]$ exists iff there is only a set of infinitary first-order formulas, up to provable equivalence in the infinitary classical deduction-system modulo $T$.
\item \thref{boolean-topos-as-boolean-classifying-topos}: every Boolean topos is the Boolean classifying topos of some geometric theory $T$.
\item \thref{omega-categorical}: this is a characterisation of those finitary first-order theories that have a coherent Boolean classifying topos. Much of this was already known from \cite{blass_boolean_1983}, but we would like to point out one interesting new consequence. Let $T$ be a finitary first-order theory (i.e.\ a first-order theory in the usual sense) in a countable language. If $T$ is $\omega$-categorical then every infinitary first-order formula is provably equivalent to a finitary first-order formula, modulo $T$ in the infinitary classical deduction-system.
\item In Section \ref{sec:relating-the-various-classifying-toposes} we relate the various notions of classifying topos by showing how to construct one from another (\thref{relating-classifying-toposes}). We also characterise when the classifying topos of a geometric theory is Boolean in \thref{for-geometric-theories}, yielding an infinitary version of the main result in \cite{blass_boolean_1983}.
\end{itemize}

\textbf{Acknowledgements.} This paper is based on the author's Master's thesis \cite{kamsma_classifying_2018}, and as such the author would like to thank the supervisor Jaap van Oosten for all his input and feedback at the time. The author would also like to thank Ming Ng for helpful discussions.
\section{The different logics}
\label{sec:the-different-logics}
We will often restrict ourselves to things of size $< \kappa$ for some cardinal $\kappa$. We will often allow for $\kappa = \infty$, which means that we remove any bounds on size.
\begin{definition}
\thlabel{first-order-formulas}
As usual, a (multi-sorted) signature $\Sigma$ (with equality) will consist of a set of sorts, relation symbols, constant symbols and function symbols. These can then be used to form terms and atomic formulas as usual. For a cardinal $\kappa$ or $\kappa = \infty$ we recall that the class of first-order $\L_{\kappa, \omega}$-formulas is the smallest class that is closed under:
\begin{enumerate}[label=(\arabic*)]
\item atomic formulas,
\item finitary conjunctions and disjunctions (including $\top$ and $\bot$),
\item implication $\to$,
\item disjunctions of size $< \kappa$ with finitely many free variables,
\item conjunctions of size $< \kappa$ with finitely many free variables,
\item existential quantification (over finite strings of variables),
\item universal quantification (over finite strings of variables).
\end{enumerate}
We will call a formula in this class a \emph{$\kappa$-first-order formula}.

The smallest class that is closed under (1), (2), (4) and (6) will be called the class of \emph{$\kappa$-geometric formulas}. The smallest class that is closed under (1), (2), (3), (4) and (6) will be called the class of \emph{$\kappa$-sub-first-order formulas}.
\end{definition}
The above definition is mostly standard (see e.g.\ \cite[Definition D1.1.3]{johnstone_sketches_2002_2}), just the notion of ``sub-first-order'' is new.
\begin{convention}
\thlabel{infty-renaming}
The different logics from \thref{first-order-formulas} will have similarly named concepts related to them (such as formulas, theories and deduction-systems). For each of these we adapt the following naming convention. We will just say ``geometric'' instead of ``$\infty$-geometric'' (e.g.\ ``geometric formula'' instead of ``$\infty$-geometric formula''). However, for ``(sub-)first-order'' this can be confusing as those traditionally refer to finitary formulas, so we will then explicitly say ``infinitary (sub-)first-order'' instead of ``$\infty$-(sub-)first-order'' (e.g.\ ``infintary first-order formula'' instead of ``$\infty$-first-order formula'').
\end{convention}
\begin{definition}
\thlabel{context}
By a \emph{context} we mean a string of variables, and its corresponding string of sorts will also be called a \emph{sort}. By a \emph{formula in context} we mean a pair $\phi(x)$, where $x$ is a string of variables and $\phi$ a formula with its free variables among those in $x$. Given a list $t$ of terms, whose sort matches that of $x$, we write $\phi[t/x]$ for the formula where each term in the list $t$ is substituted for the corresponding free variable in $x$ in $\phi$.
\end{definition}
We recall the basics of categorical semantics and establish terminology and notation along the way.
\begin{definition}
\thlabel{regular-category}
A category $\C$ is called \emph{regular} if it is finitely complete, has coequalisers of kernel pairs and regular epimorphisms are stable under pullback.
\end{definition}
We have taken the definition from \cite{butz_regular_1998}. There is a different treatment in \cite[Section A1.3]{johnstone_sketches_2002_1}. The important part is that for any $f: X \to Y$ in a regular category we can define a functor $f^*: \Sub(Y) \to \Sub(X)$ by pulling back along $f$. Furthermore, this functor has a left adjoint $\exists_f: \Sub(X) \to \Sub(Y)$, arising from the (regular epi, mono)-factorisation system in a regular category.
\begin{definition}
\thlabel{geometric-heyting-boolean-category}
Let $\kappa$ be a cardinal or $\infty$. A well-powered regular category $\C$ is called:
\begin{itemize}
\item \emph{$\kappa$-geometric} if it is regular and subobject posets have joins of size $<\kappa$, which are stable under pullback;
\item \emph{$\kappa$-sub-Heyting} if it is $\kappa$-geometric where the subobject posets are Heyting algebras and for any arrow $f: X \to Y$ in $\C$ the operation $f^*: \Sub(Y) \to\Sub(X)$ is a Heyting algebra homomorphism;
\item \emph{$\kappa$-Heyting} if it is $\kappa$-sub-Heyting, the subobject posets have meets of size $< \kappa$, which are stable under pullback, and for any arrow $f: X \to Y$ in $\C$ the functor $f^*: \Sub(Y) \to \Sub(X)$ has a right adjoint $\forall_f: \Sub(X) \to \Sub(Y)$;
\item \emph{$\kappa$-Boolean} if it is $\kappa$-Heyting and the subobject posets are Boolean algebras.
\end{itemize}
\end{definition}
Each kind of category from \thref{geometric-heyting-boolean-category} allows for interpretation of the similarly named logic. We assume the reader is familiar with categorical semantics (see e.g.\ \cite[Section D1.2]{johnstone_sketches_2002_2}) and we merely establish notation.
\begin{definition}
\thlabel{categorical-semantics}
Let $\kappa$ be a cardinal or $\infty$, let $\C$ be a $\kappa$-geometric (resp.\ $\kappa$-(sub-)Heyting) category and let $\Sigma$ be a signature. Let $M$ be a \emph{$\Sigma$-structure} in $\C$. For a sort $X$ we write $X^M$ for its interpretation in $M$ (recall we also use the term sort for a, possibly empty, string of sorts). A $\kappa$-geometric (resp.\ $\kappa$-(sub-)first-order) formula in context $\phi(x)$ will then be interpreted as a subobject $\{x : \phi(x)\}^M$ of $X^M$, where $X$ is the sort of $x$.
\end{definition}
\begin{definition}
\thlabel{geometric-heyting-boolean-functor}
Let $\kappa$ be a cardinal or $\infty$.
\begin{itemize}
\item A functor $F: \C \to \D$ between $\kappa$-geometric categories is called \emph{$\kappa$-geometric} if it preserves finite limits, regular epimorphisms and joins of size $< \kappa$.
\item A functor $F: \C \to \D$ between $\kappa$-sub-Heyting categories is called \emph{$\kappa$-sub-Heyting} if is $\kappa$-geometric and preserves Heyting implication on the subobject posets.
\item A functor $F: \C \to \D$ between $\kappa$-Heyting categories is called \emph{$\kappa$-Heyting} if is $\kappa$-geometric, preserves meets of size $< \kappa$ and commutes with the $\forall_{(-)}$ operation.
\end{itemize}
We write $\Geom_\kappa(\C, \D)$ (resp.\ $\SubHeyt_\kappa(\C, \D)$ or $\Heyt_\kappa(\C, \D)$) for the category of $\kappa$-geometric (resp.\ $\kappa$-sub-Heyting or $\kappa$-Heyting) functors between $\C$ and $\D$, with natural transformations as arrows.
\end{definition}
\begin{example}
\thlabel{topos-infty-heyting}
Any topos $\E$ is an $\infty$-Heyting category and it is $\infty$-Boolean iff $\E$ is a Boolean topos. The inverse image part of a geometric morphism between toposes is a geometric functor. Geometric functors whose inverse image part is $\infty$-(sub-)-Heyting will be considered in \thref{kinds-of-geometric-morphisms}.
\end{example}
\begin{remark}
\thlabel{geometric-heyting-boolean-functor-remarks}
We make a few remarks about \thref{geometric-heyting-boolean-functor}.
\begin{enumerate}[label=(\roman*)]
\item Being a $\kappa$-geometric functor implies that the functor $F$ commutes with the $\exists_{(-)}$ operation. That is, for any $f: X \to Y$ in $\C$ and $A \leq X$ we have $F(\exists_f(A)) = \exists_{F(f)}(F(A))$.
\item Any $\kappa$-Heyting functor is $\kappa$-sub-Heyting, because the implication operation can be built from the $\forall_{(-)}$ operation as follows. Let $A, B \leq X$ and denote by $a$ a monomorphism representing $A$. Then $A \to B = \forall_a(A \wedge B)$.
\item By a functor $F: \C \to \D$ that ``commutes with the $\forall_{(-)}$ operation'' we mean, like in point (i), that for all $f: X \to Y$ in $\C$ and $A \leq X$ we have $F(\forall_f(A)) = \forall_{F(f)}(F(A))$.
\end{enumerate}
\end{remark}
\begin{fact}[{\cite[Lemma D1.2.13]{johnstone_sketches_2002_2}}]
\thlabel{functor-preserves-formulas}
Let $\kappa$ be a cardinal or $\infty$ and let $\Sigma$ be a signature. Suppose that $F: \C \to \D$ is a $\kappa$-geometric (resp.\ $\kappa$-(sub-)Heyting) functor between $\kappa$-geometric (resp.\ $\kappa$-(sub-)Heyting) categories. Then given a $\Sigma$-structure $M$ in $\C$ we naturally find a $\Sigma$-structure $F(M)$ in $\D$ and $F$ preserves interpretations of $\kappa$-geometric (resp.\ $\kappa$-(sub-)first-order) formulas. That is, for any such formula in context $\phi(x)$ we have
\[
F(\{x : \phi(x)\}^M) = \{x : \phi(x)\}^{F(M)}.
\]
\end{fact}
\begin{proposition}
\thlabel{kappa-geometric-from-boolean-is-kappa-heyting}
Let $\kappa$ be a cardinal or $\infty$. Any $\kappa$-geometric functor $F: \C \to \D$ from a $\kappa$-Boolean category $\C$ to a $\kappa$-Heyting category $\D$ is a $\kappa$-Heyting functor.
\end{proposition}
\begin{proof}
Let $A \leq X$ be a subobject in $\C$. Then $F(A) \vee F(\neg A) = F(A \vee \neg A) = F(X)$ and $F(A) \wedge F(\neg A) = F(A \wedge \neg A) = F(0) = 0$, so we have $F(\neg A) = \neg F(A)$.

We first check that $F$ preserves meets of size $< \kappa$. Let $\{A_i\}_{i \in I}$ be a set of subobjects of $X$ in $\C$ with $|I| < \kappa$. Note that $F(\bigwedge_{i \in I} A_i) \leq F(A_i)$ for all $i \in I$ and so $F(\bigwedge_{i \in I} A_i) \leq \bigwedge_{i \in I} F(A_i)$. For the other direction we use that $F$ preserves complements and joins:
\[
\bigwedge_{i \in I} F(A_i) \leq
\neg \bigvee_{i \in I} \neg F(A_i) =
F \left( \neg \bigvee_{i \in I} \neg A_i \right) =
F \left( \bigwedge_{i \in I} A_i \right),
\]
where the first inequality is generally true in any Heyting algebra.

We are left to check that for any arrow $f: X \to Y$ in $\C$ the functor $F$ preserves $\forall_f$. So let $A \leq X$ and note that
\[
F(f)^*(F(\forall_f(A))) = F(f^*(\forall_f(A))) \leq F(A),
\]
where the equality is because $F$ preserves pullbacks and the inequality is $F$ applied to the co-unit of $f^* \dashv \forall_f$. So $F(\forall_f(A)) \leq \forall_{F(f)}(F(A))$. For the other direction we use that $F$ preserves complements and images:
\[
\forall_{F(f)}(F(A)) \leq
\neg \exists_{F(f)} \neg F(A) =
F(\neg \exists_f(\neg A)) =
F(\forall_f(A)),
\]
where the first inequality is a general fact for adjoints between Heyting algebras.
\end{proof}
Suppose that $M$ and $N$ are $\Sigma$-structures in some $\kappa$-geometric category $\C$ and that for each sort $X$ of $\Sigma$ we have an arrow $h_X: X^M \to X^N$ such that whenever $X = S_1, \ldots, S_n$ is a list of sorts we have $h_X = h_{S_1} \times \ldots \times h_{S_n}$. We recall from \cite[Lemma D1.2.9]{johnstone_sketches_2002_2} that $h: M \to N$ is a \emph{homomorphism} if for all $\kappa$-geometric formulas $\phi(x)$ in context we have
\[
\{x : \phi(x)\}^M \leq h_X^*(\{x : \phi(x)\}^N).
\]
Following \cite[Section 2]{butz_classifying_1998} we define stronger notions of morphisms between structures, depending on the fragment of logic they preserve.
\begin{definition}
\thlabel{elementary-morphisms}
Let $\kappa$ be a cardinal or $\infty$, let $\C$ be a $\kappa$-sub-Heyting (resp.\ $\kappa$-Heyting) category and let $\Sigma$ be a signature. We call a homomorphism $h: M \to N$ of $\Sigma$-structures in $\C$ a \emph{$\kappa$-sub-elementary morphism} (resp.\ \emph{$\kappa$-elementary morphism}) if for all $\kappa$-sub-first-order (resp.\ $\kappa$-first-order) formulas in context $\phi(x)$ we have
\[
\{ x : \phi(x) \}^M \leq h_X^*(\{ x : \phi(x) \}^N).
\]
We write $\SStr(\C)$ (resp.\ $\SStr(\C)_{\kappa \to}$ or $\SStr(\C)_{\kappa}$) for the category of all $\Sigma$-structures in $\C$ with homomorphisms (resp.\ $\kappa$-sub-elementary or $\kappa$-elementary) embeddings between them.
\end{definition}
Note that $h$ is a homomorphism iff the inequality in the above definition holds for all atomic formulas. There is thus no real dependence on $\kappa$, which is why there is no subscript in $\SStr(\C)$. However, the notion of $\kappa$-(sub-)elementary morphism does depend on $\kappa$, and so the subscript matters: \cite[Example 2.2]{butz_classifying_1998} gives an example of an $\omega$-elementary morphism that is not $\omega_1$-elementary. We also refer to \cite[Example 2.1]{butz_classifying_1998} to note that an $\infty$-elementary morphism need not be injective, which explains why they are not named ``$\infty$-elementary embeddings''.
\begin{definition}
\thlabel{sequents-and-theories}
Let $\kappa$ be a cardinal or $\infty$. A \emph{$\kappa$-geometric} (resp.\ \emph{$\kappa$-(sub-)first-order}) \emph{sequent} is an expression of the form $\phi(x) \vdash_x \psi(x)$, where $x$ is a context and $\phi(x)$ and $\psi(x)$ $\kappa$-geometric (resp.\ $\kappa$-(sub-)first-order) are formulas in context $x$. A \emph{$\kappa$-geometric} (resp.\ \emph{$\kappa$-(sub-)first-order}) \emph{theory} is a set of $\kappa$-geometric (resp.\ $\kappa$-(sub-)first-order) sequents.

We say that a sequent $\phi(x) \vdash_x \psi(x)$ is \emph{valid} in a structure $M$ if $\{x : \phi(x)\}^M \leq \{x : \psi(x)\}^M$. We say that $M$ is a \emph{model} of a theory $T$ if every sequent in $T$ is valid in $M$. For an appropriate category $\C$ we write $\TMod(\C)$ (resp.\ $\TMod(\C)_{\kappa \to}$ or $\TMod(\C)_\kappa$) for the full subcategory of $\SStr(\C)$ (resp.\ $\SStr(\C)_{\kappa \to}$ or $\SStr(\C)_{\kappa}$) whose objects are models of $T$.
\end{definition}
\begin{proposition}
\thlabel{functor-preserves-models}
Let $\kappa$ be a cardinal or $\infty$ and let $F: \C \to \D$ be a functor.
\begin{itemize}
\item If $F$ is $\kappa$-geometric then for any $\kappa$-geometric theory $T$ it induces a functor $F: \TMod(\C) \to \TMod(\D)$.
\item If $F$ is $\kappa$-sub-Heyting then for any $\kappa$-sub-first-order theory $T$ it induces a functor $F: \TMod(\C)_{\kappa \to} \to \TMod(\D)_{\kappa \to}$.
\item If $F$ is $\kappa$-Heyting then for any $\kappa$-first-order theory $T$ it induces a functor $F: \TMod(\C)_\kappa \to \TMod(\D)_\kappa$.
\end{itemize}
\end{proposition}
\begin{proof}
This is just writing out definitions, using \thref{functor-preserves-formulas}.
\end{proof}

Each of the different logics allows for a sequent-style deduction-system with rules for each of its allowed connectives. This is also where yet another logic appears: classical infinitary first-order logic. So far the distinction between the different logic has been in their expressive power, but their intended deductive power is intuitionistic (as that is the internal logic of toposes). Classical infinitary first-order logic then distinguishes itself by having greater deductive power.
\begin{definition}
\thlabel{deduction-systems}
Let $\kappa$ be a cardinal or $\infty$. Referring to the deduction rules outlined in \cite[Definition D1.3.1]{johnstone_sketches_2002_2}, we define the follows deduction-systems.
\begin{itemize}
\item The \emph{$\kappa$-geometric deduction-system} contains the structural rules (a), equality rules (b), rules for finitary conjunction (c) and for infinitary disjunction for disjunctions of size $< \kappa$ (h), rules for existential quantification (f) and the infinitary distributive axiom and Frobenius axiom (i).
\item The \emph{$\kappa$-sub-first-order deduction-system} expands the $\kappa$-geometric deduction-system by also allowing the rules for implication (e).
\item The \emph{$\kappa$-first-order deduction-system} expands the $\kappa$-sub-first-order deduction-system by also allowing the rules for infinitary conjunction for conjunctions of size $< \kappa$ (h) and rules for universal quantification (g).
\item The \emph{$\kappa$-classical deduction-system} expands the $\kappa$-first-order deduction-system by allowing the law of excluded middle (or equivalently: double negation elimination), see e.g.\ \cite[Remark D1.3.3]{johnstone_sketches_2002_2}.
\end{itemize}
\end{definition}
\begin{fact}[Soundness, {\cite[Proposition D1.3.2]{johnstone_sketches_2002_2}}]
\thlabel{soundness}
Let $\kappa$ be a cardinal or $\infty$. The $\kappa$-geometric (resp.\ $\kappa$-(sub-)first-order or $\kappa$-classical) deduction-system is sound for $\kappa$-geometric (resp.\ $\kappa$-(sub-)Heyting or $\kappa$-Boolean) categories. That is, if $M$ is a structure in such a category that satisfies a set of sequents $T$ and the sequent $\sigma$ is deducible from $T$ using the relevant deduction-system then $\sigma$ is valid in $M$.
\end{fact}
The proof of \thref{soundness} is by induction on the proof tree. This is why we restricted our deduction-systems to disjunctions (and conjunctions) of size $< \kappa$, because otherwise the relevant induction step cannot be executed in a $\kappa$-geometric (or $\kappa$-Heyting / $\kappa$-Boolean) category. We will later see that this is not a genuine restriction on the strength of the deduction-system in \thref{kappa-deduction-is-infinitary}.
\section{Definitions of classifying toposes}
\label{sec:definitions-of-classifying-toposes}
The definition below is not the original definition \cite[Definition 1.1]{johnstone_open_1980}. We use equivalent characterisations in logical terms that are directly useful to us (see \cite[Lemma 3.1 and Theorem 3.2]{johnstone_open_1980}).
\begin{definition}
\thlabel{kinds-of-geometric-morphisms}
For a geometric morphism $f: \E \to \F$ we use the standard notation of $f^*: \F \to \E$ for its inverse image part (we will have no use for the direct image parts in this paper). We call such a geometric morphism:
\begin{itemize}
\item \emph{sub-open} if it preserves implication on the subobject posets,
\item \emph{open} if it commutes with the $\forall_{(-)}$ operation.
\end{itemize}
We write $\Topos(\E, \F)$ for the category of geometric morphisms from $\E$ to $\F$, with as arrows natural transformations between the inverse image parts. We write $\SubOpen(\E, \F)$ and $\Open(\E, \F)$ for the full subcategories of sub-open and open geometric morphisms respectively.
\end{definition}
It follows that every open geometric morphism preserves infinitary meets of subobjects as well (see e.g.\ \cite[Theorem IX.6.3]{maclane_sheaves_1994}).
\begin{fact}
\thlabel{open-sub-open-geometric-morphism-is-heyting-sub-heyting}
Let $f: \E \to \F$ be a (sub-)open geometric morphism. Then its inverse image part $f^*$ is $\infty$-(sub-)Heyting.
\end{fact}
\begin{fact}[{\cite[Corollary 3.5]{johnstone_open_1980}}]
\thlabel{boolean-iff-every-map-into-is-open}
The following are equivalent for a topos $\E$:
\begin{enumerate}[label=(\roman*)]
\item $\E$ is Boolean,
\item every geometric morphism into $\E$ is open,
\item every geometric morphism into $\E$ is sub-open.
\end{enumerate}
\end{fact}
We note that \thref{kappa-geometric-from-boolean-is-kappa-heyting} is an alternative, logic-based, proof of (i) $\Rightarrow$ (ii) in the above fact.
\begin{definition}
\thlabel{classifying-toposes-def}
We define four different notions of classifying topos.
\begin{itemize}
\item For a geometric theory $T$, the \emph{classifying topos} $\Set[T]$ is a topos such that for every topos $\E$ there is an equivalence
\[
\Topos(\E, \Set[T]) \simeq \TMod(\E).
\]
\item For an infinitary sub-first-order theory $T$, the \emph{sub-first-order classifying topos} $\SetSFO[T]$, if it exists, is a topos such that for every topos $\E$ there is an equivalence
\[
\SubOpen(\E, \SetSFO[T]) \simeq \TMod(\E)_{\infty \to}.
\]
\item For an infinitary first-order theory $T$, the \emph{first-order classifying topos} $\SetFO[T]$, if it exists, is a topos such that for every topos $\E$ there is an equivalence
\[
\SubOpen(\E, \SetFO[T]) \simeq \TMod(\E)_\infty.
\]
\item For an infinitary first-order theory $T$, the \emph{Boolean classifying topos} $\SetB[T]$, if it exists, is a Boolean topos such that for every Boolean topos $\E$ there is an equivalence
\[
\Topos(\E, \SetB[T]) \simeq \TMod(\E)_\infty.
\]
\end{itemize}
Furthermore, in each case we require there to be a model $G_T$ of $T$ in the appropriate classifying topos, called the \emph{generic model}, such that the equivalence is given by sending a (sub-open / open) geometric morphism $f$ to the model $f^*(G_T)$ and a natural transformation $\eta: f \to g$ is sent to the homomorphism $f^*(G_T) \to g^*(G_T)$ whose component at sort $X$ is $\eta_{X^{G_T}}: X^{f^*(G_T)} \to X^{g^*(G_T)}$.
\end{definition}
\begin{remark}
\thlabel{classifying-toposes-def-remarks}
We make some remarks about \thref{classifying-toposes-def}.
\begin{enumerate}[label=(\roman*)]
\item The definition of a classifying topos for geometric theories is standard (see e.g.,\ \cite[Proposition D3.1.12]{johnstone_sketches_2002_2} or \cite[Theorem 2.1.8]{caramello_theories_2017}). The definition of a first-order classifying topos is that from \cite{butz_classifying_1998}.
\item For geometric theories we know that the classifying topos always exists. For each of the other kinds this may not be the case, as emphasised in the definition, see \thref{classifying-toposes-counterexamples}.
\item In the definition of the Boolean classifying topos we consider geometric morphisms. Since geometric morphisms into a Boolean topos are always open (\thref{boolean-iff-every-map-into-is-open}), we can equivalently define the Boolean classifying topos $\SetB[T]$ as one such that for every Boolean topos we have:
\[
\Open(\E, \SetB[T]) \simeq \TMod(\E)_\infty.
\]
\item We speak about ``the'' (Boolean or (sub-)first-order) classifying topos because it straightforwardly follows from the definition that such a topos is unique up to equivalence (if it exists).
\item The requirement at the end of the definition is equivalent to requiring that the equivalences are natural in a 2-categorical sense, see \cite[Remark 2.1.7(a)]{caramello_theories_2017}.
\end{enumerate}
\end{remark}
As pointed out in \cite{butz_classifying_1998}, the problem with the existence of a first-order classifying topos is that there can be too many infinitary first-order formulas (up to equivalence). The same kind of problem arises for sub-first-order and Boolean classifying toposes. We expand their example to these cases.
\begin{fact}[\cite{de_jongh_class_1980}]
\thlabel{proper-class-of-intuitionistic-formulas}
There is a proper class of propositional infinitary sub-first-order formulas on two propositional variables, such that no two of these formulas are equivalent modulo the infinitary first-order deduction-system.
\end{fact}
\begin{fact}[{\cite{gaifman_infinite_1964, hales_non-existence_1964}, with a simpler proof in \cite{solovay_new_1966}}]
\thlabel{proper-class-of-classical-formulas}
There is a proper class of propositional infinitary first-order formulas on a countable infinity of propositional variables, such that no two of these formulas are equivalent modulo the infinitary classical deduction-system.
\end{fact}
\begin{example}
\thlabel{classifying-toposes-counterexamples}
We give examples of theories for which the sub-first-order, first-order and Boolean classifying toposes cannot exist.
\begin{enumerate}[label=(\roman*)]
\item Let $T$ be the empty theory in the signature that consists of two propositional variables. Suppose, for a contradiction, that $\SetSFO[T]$ exists. Let $\kappa$ be a cardinal and let $H_\kappa$ be a complete Heyting algebra on two generators containing at least $\kappa$ many inequivalent (in the infinitary first-order deduction-system) sub-first-order formulas. Such a Heyting algebra exists by \thref{proper-class-of-intuitionistic-formulas}. Considering $H_\kappa$ as a locale, we obtain the localic topos $\Sh(H_\kappa)$. In this topos we have that $\Sub(1) \cong H_\kappa$, so it contains $H_\kappa$ as a model of $T$. That means there must be a sub-open-geometric morphism $f_\kappa: \Sh(H_\kappa) \to \SetSFO[T]$ such that $f_\kappa^*(G_T) = H_\kappa$, and in particular we find a function $f_\kappa^*: \Sub_{\SetSFO[T]}(1) \to H_\kappa$ that preserves interpretations of infinitary sub-first-order propositional formulas on our two propositional variables. By our choice of $H_\kappa$, the image of $f_\kappa^*: \Sub_{\SetSFO[T]}(1) \to H_\kappa$ must at least have cardinality $\kappa$, so we have that $\Sub_{\SetSFO[T]}(1)$ must have at least cardinality $\kappa$. As $\kappa$ was arbitrary, we find that $\Sub_{\SetSFO[T]}(1)$ is a proper class, which it cannot be.
\item The case for $\SetFO[T]$ is analogous to point (i), and is exactly the example considered in \cite{butz_classifying_1998}.
\item Take $T$ to be the empty theory in the signature that consists of a countable infinity of propositional variables. Then $\SetB[T]$ does not exist. The proof of this is analogous to point (i), only we use \thref{proper-class-of-classical-formulas} instead of \thref{proper-class-of-intuitionistic-formulas}. Two more details should be noted. Firstly, we can take $H_\kappa$ to be a complete Boolean algebra instead, ensuring that $\Sh(H_\kappa)$ is Boolean. Secondly, the geometric morphisms involved are open (see e.g.\ \thref{classifying-toposes-def-remarks}(iii)), so their inverse image parts preserve infinitary first-order formulas.
\end{enumerate}
\end{example}
We finish this section by recalling a simplified version of the continuous version of Diaconecu's theorem, which will be essential in constructing the various kinds of classifying toposes (see e.g.\ \cite[Corollary VII.7.4]{maclane_sheaves_1994}).
\begin{fact}[Continuous version of Diaconescu's theorem]
\thlabel{continuous-diaconescu}
Let $\C$ be a small category with finite limits and let $J$ be a Grothendieck topology on $\C$. Then there is an equivalence of categories
\[
\Topos(\E, \Sh(\C, J)) \simeq \FlatCon((\C, J), \E).
\]
Here $\FlatCon((\C, J), \E)$ is the category of finite limit preserving functors $\C \to \E$ that are $J$-continuous (i.e.\ sending covering sieves to epimorphic families).

Furthermore, if $J$ is subcanonical (i.e.\ representable presheaves are sheaves) then this correspondence is given by sending a geometric morphism $f$ to $f^* y$, where $y: \C \to \Sh(\C, J)$ is the Yoneda embedding.
\end{fact}
\section{Syntactic categories}
\label{sec:syntactic-categories}
The construction of syntactic categories is standard, see e.g.\ \cite[Section D1.4]{johnstone_sketches_2002_2}. We recall the construction here, mainly to establish our notation and terminology. We first give the general recipe below and then make it precise in \thref{syntactic-categories} for the various logics that we consider.

Fix a theory $T$, a class $F$ of formulas and a deduction-system $S$. A syntactic category for $T$ has as objects equivalence classes $[\phi(x)]$ of formulas in context, where the formulas are taken from $F$. Two such formulas $\phi(x)$ and $\phi'(x')$ are considered to be equivalent if $x$ and $x'$ are of the same sort and $\phi = \phi'[x/x']$. In defining arrows between $[\phi(x)]$ and $[\psi(y)]$ we can then assume that $x$ and $y$ have no variables in common. An arrow $[\phi(x)] \to [\psi(y)]$ is then a $T$-provable equivalence class $[\theta(x, y)]$ of formulas (from $F$) that are $T$-provably functional from $\phi(x)$ to $\psi(y)$, where provability is taken with respect to $S$.
\begin{definition}
\thlabel{syntactic-categories}
Let $\kappa$ be a cardinal or $\infty$. Following the recipe above, we define the following \emph{syntactic categories}.
\begin{itemize}
\item For a $\kappa$-geometric theory $T$, we let $\SynG_\kappa(T)$ be the syntactic category of $\kappa$-geometric formulas with respect to the $\kappa$-geometric deduction-system.
\item For a $\kappa$-sub-first-order theory $T$, we let $\SynSFO_\kappa(T)$ be the syntactic category of $\kappa$-sub-first-order formulas in with respect to the $\kappa$-sub-first-order deduction-system.
\item For a $\kappa$-first-order theory $T$, we let $\SynFO_\kappa(T)$ be the syntactic category of $\kappa$-first-order formulas in with respect to the $\kappa$-first-order deduction-system.
\item For a $\kappa$-first-order theory $T$, we let $\SynB_\kappa(T)$ be the syntactic category of $\kappa$-first-order formulas in with respect to the $\kappa$-classical deduction-system.
\end{itemize}
\end{definition}
\begin{proposition}
\thlabel{types-of-syntactic-categories}
Let $\kappa$ be a cardinal.
\begin{enumerate}[label=(\roman*)]
\item For any $\kappa$-geometric theory $T$ the category $\SynG_\kappa(T)$ is $\kappa$-geometric.
\item For any $\kappa$-sub-first-order theory $T$ the category $\SynSFO_\kappa(T)$ is $\kappa$-sub-Heyting.
\item For any $\kappa$-first-order theory $T$ the category $\SynFO_\kappa(T)$ is $\kappa$-Heyting.
\item For any $\kappa$-first-order theory $T$ the category $\SynB_\kappa(T)$ is $\kappa$-Boolean.
\end{enumerate}
\end{proposition}
\begin{proof}
Straightforward, see for example \cite[Lemma D1.4.10]{johnstone_sketches_2002_2}. For a detailed approach that is more closely matching the terminology here, see \cite[Chapters 5 and 7]{kamsma_classifying_2018}. For well-poweredness it is important that we assumed $\kappa$ to be a cardinal, and not $\infty$.
\end{proof}
Let $\kappa$ be a cardinal. For $T$ and $\C$ a theory and its corresponding syntactic category as in \thref{syntactic-categories} there is a canonical structure $U^\kappa_T$ in $\C$ by interpreting a sort $X$ as $[\top(x)]$, where $x$ is a variable that sort. Function symbols $f: X \to Y$ and relation symbols $R(x)$ are then naturally interpreted as $[f(x) = y]$ and $[R(x)]$ respectively. For any formula $\phi(x)$ in the appropriate logic we then have
\[
\{ x : \phi(x) \}^{U^\kappa_T} = [\phi(x)].
\]
In particular, $U^\kappa_T$ is a model of $T$.
\begin{definition}
\thlabel{universal-syntactic-model}
We call the model $U^\kappa_T$ the \emph{universal syntactic model} of $T$.
\end{definition}
\begin{proposition}
\thlabel{sequents-in-universal-syntactic-model}
Let $\kappa$ be a cardinal and let $T$ be a theory in one of the four logics considered in \thref{syntactic-categories}. Then a sequent $\sigma$ is valid in $U^\kappa_T$, as a model in the appropriate syntactic category, if and only if $\sigma$ can be deduced from $T$ in the appropriate deduction-system.
\end{proposition}
\begin{proof}
By construction.
\end{proof}
\begin{proposition}
\thlabel{functors-as-models}
Let $\kappa$ be a cardinal.
\begin{enumerate}[label=(\roman*)]
\item For any $\kappa$-geometric theory $T$ and $\kappa$-geometric category $\C$ there is an equivalence $\Geom_\kappa(\SynG_\kappa(T), \C) \simeq \TMod(\C)$.
\item For any $\kappa$-sub-first-order theory $T$ and $\kappa$-sub-Heyting category $\C$ there is an equivalence $\SubHeyt_\kappa(\SynSFO_\kappa(T), \C) \simeq \TMod(\C)_{\kappa \to}$.
\item For any $\kappa$-first-order theory $T$ and $\kappa$-Heyting category $\C$ there is an equivalence $\Heyt_\kappa(\SynFO_\kappa(T), \C) \simeq \TMod(\C)_\kappa$.
\item For any $\kappa$-first-order theory $T$ and $\kappa$-Boolean category $\C$ there is an equivalence $\Heyt_\kappa(\SynB_\kappa(T), \C) \simeq \TMod(\C)_\kappa$.
\end{enumerate}
Each of these equivalences is given by sending a functor $F$ to the model $F(U^\kappa_T)$, and a natural transformation $\eta: F \to G$ is sent to the homomorphism $F(U^\kappa_T) \to G(U^\kappa_T)$ whose component at sort $X$ is $\eta_{X^{U^\kappa_T}}: X^{F(U^\kappa_T)} \to X^{G(U^\kappa_T)}$.
\end{proposition}
\begin{proof}
Clearly $F(U^\kappa_T)$ is a model of $T$ again. Conversely, for a model $M$ of $T$ in $\C$ we can define a functor $F$ from the relevant syntactic category into $\C$ by setting $F([\phi(x)]) = \{x : \phi(x)\}^M$ and for an arrow $[\theta(x, y)]$ in the syntactic category we let $F([\theta(x, y)])$ be the arrow in $\C$ corresponding to the graph $\{x,y : \theta(x,y)\}^M$. It remains to check that natural transformations of functors on the left-hand side correspond to homomorphisms or $\kappa$-(sub-)elementary morphisms. Indeed, a natural transformation from $F$ to $G$ corresponds to commuting diagrams like below (where $\phi(x)$ ranges over formulas in the relevant logic):
\[
\begin{tikzcd}
\{x : \phi(x)\}^{F(U^\kappa_T)} \arrow[rr] \arrow[d, hook] &  & \{x : \phi(x)\}^{G(U^\kappa_T)} \arrow[d, hook] \\
X^{F(U^\kappa_T)} \arrow[rr, "h_X"']                       &  & X^{G(U^\kappa_T)}                              
\end{tikzcd}
\]
Which in turn corresponds to the condition $\{x : \phi(x)\}^{F(U^\kappa_T)} \leq h_X^*(\{x : \phi(x)\}^{G(U^\kappa_T)})$.
\end{proof}
\begin{corollary}
\thlabel{geometric-functors-as-boolean-models}
The equivalence in \thref{functors-as-models}(iv) can be simplified as follows. For any $\kappa$-first-order theory $T$ and $\kappa$-Boolean category $\C$ there is an equivalence $\Geom_\kappa(\SynB_\kappa(T), \C) \simeq \TMod(\C)_\kappa$.
\end{corollary}
\begin{proof}
This follows from \thref{kappa-geometric-from-boolean-is-kappa-heyting} because $\SynB_\kappa(T)$ is $\kappa$-Boolean, so we have $\Geom_\kappa(\SynB_\kappa(T), \C) = \Heyt_\kappa(\SynB_\kappa(T), \C)$.
\end{proof}
\begin{corollary}
\thlabel{boolean-classifying-topos-iff-homomorphisms}
A Boolean topos $\F$ is the Boolean classifying topos of an infinitary first-order theory $T$ if for all Boolean toposes $\E$ there is an equivalence
\[
\Topos(\E, \F) \simeq \TMod(\E),
\]
where this equivalence is given by a generic model in the same way as \thref{classifying-toposes-def}.
\end{corollary}
\begin{proof}
The difference with \thref{classifying-toposes-def} is that we have $\TMod(\E)$ on the right-hand side instead of $\TMod(\E)_\infty$. Any homomorphism $h: M \to N$ of models of $T$ in $\E$ corresponds to a natural transformation of (the inverse image parts of) geometric morphisms into $\F$. Since $\F$ is Boolean these geometric morphisms are open (\thref{boolean-iff-every-map-into-is-open}). In particular, their inverse image parts are $\infty$-Heyting functors. So by (the proof of, with $G_T$ in the role of $U^\kappa_T$) \thref{functors-as-models} we have that $h$ is an $\infty$-elementary morphism. We conclude that $\TMod(\E) = \TMod(\E)_\infty$, and we have matched \thref{classifying-toposes-def}.
\end{proof}

\section{Completeness theorems}
\label{sec:completeness-theorems}
\begin{definition}
\thlabel{kappa-covering-topology}
Let $\kappa$ be a cardinal and let $\C$ be a $\kappa$-geometric category. A \emph{$\kappa$-covering family} for an object $Y$ in $\C$ is a set $\{f_i: X_i \to Y\}_{i \in I}$ of arrows in $\C$ with $|I| < \kappa$ such that $\bigvee_{i \in I} \Im f_i = Y$. We define the \emph{$\kappa$-covering topology} $J_\kappa$ on $\C$ by saying that a sieve is covering if it contains a $\kappa$-covering family.
\end{definition}
\begin{fact}[{\cite[Example C2.1.12(d)]{johnstone_sketches_2002_2} or \cite[Lemma X.5.4]{maclane_sheaves_1994}}]
\thlabel{kappa-covering-subcanonical}
Let $\kappa$ be a cardinal and let $\C$ be a $\kappa$-geometric category. The $\kappa$-covering topology on $\C$ is a subcanonical Grothendieck topology on $\C$. That is, representable presheaves are sheaves.
\end{fact}
Given \thref{kappa-covering-subcanonical} we can, and will, treat the Yoneda embedding $y: \C \to \Set^{\C^\op}$ as having codomain $\Sh(\C, J_\kappa)$.
\begin{proposition}
\thlabel{flat-continuous-is-geometric}
Let $\kappa$ be a cardinal and let $\C$ be a $\kappa$-geometric category. Then for any topos $\E$ we have
\[
\FlatCon((\C, J_\kappa), \E) = \Geom_\kappa(\C, \E).
\]
\end{proposition}
\begin{proof}
Since $\C$ is finitely complete we have that a functor $F: \C \to \E$ is flat iff it preserves finite limits. Furthermore, $J_\kappa$-continuity is exactly saying that $F$ preserves images and joins of size $< \kappa$.
\end{proof}
\begin{proposition}
\thlabel{yoneda-preserves-implication}
Let $\kappa$ be a cardinal and let $\C$ be a $\kappa$-sub-Heyting category. Then the Yoneda embedding $y: \C \to \Sh(\C, J_\kappa)$ preserves Heyting implications.
\end{proposition}
\begin{proof}
Let $A, B \leq X$ be subobjects in $\C$. We need to show that for any $S \leq y(X)$ we have that $S \leq y(A \to B)$ iff $S \wedge y(A) \leq y(B)$. Note that subobjects of $y(X)$ in $\Sh(\C, J_\kappa)$ correspond to $J_\kappa$-closed sieves on $X$. For any arrow $f: C \to X$ in $S$ we have that the monomorphism $\Im(f) \to X$ is in $S$ as well, so $S = \bigvee \{ y(C) : y(C) \leq S \}$. So we indeed have
\begin{align*}
S &\leq y(A \to B) & &\Longleftrightarrow \\
y(C) &\leq y(A \to B) &\text{ for all } y(C) \leq S \quad &\Longleftrightarrow \\
C &\leq A \to B &\text{ for all } y(C) \leq S \quad &\Longleftrightarrow \\
C \wedge A &\leq B &\text{ for all } y(C) \leq S \quad &\Longleftrightarrow \\
y(C \wedge A) = y(C) \wedge y(A) &\leq y(B) &\text{ for all } y(C) \leq S \quad &\Longleftrightarrow \\
S \wedge y(A) &\leq y(B), &&
\end{align*}
where the final line follows because $S \wedge y(A) = \bigvee \{ y(C) \wedge y(A) : y(C) \leq S \}$.
\end{proof}
\begin{corollary}
\thlabel{yoneda-geometric-sub-heyting}
Let $\kappa$ be a cardinal and let $\C$ be a $\kappa$-geometric (resp.\ $\kappa$-(sub-)Heyting) category. Then $y: \C \to \Sh(\C, J_\kappa)$ is $\kappa$-geometric (resp.\ $\kappa$-(sub-)Heyting).
\end{corollary}
\begin{proof}
By the continuous version of Diaconescu's theorem \thref{continuous-diaconescu} we have that $y$, which corresponds to the identity on $\Sh(\C, J_\kappa)$, is flat and $J_\kappa$-continuous, and hence it is $\kappa$-geometric by \thref{flat-continuous-is-geometric}. The claim about being $\kappa$-sub-Heyting then follows from \thref{yoneda-preserves-implication}, where the one about being $\kappa$-Heyting is precisely \cite[Corollary 3.2]{butz_classifying_1998}.
\end{proof}
\begin{corollary}
\thlabel{validity-in-yoneda-universal-model}
Let $\kappa$ be a cardinal and let $T$ be a $\kappa$-geometric (resp.\ $\kappa$-(sub-)first-order) theory. Write $\C$ for the appropriate syntactic category of $T$. Then $y(U^\kappa_T)$ in $\Sh(\C, J_\kappa)$ is a model of $T$. Furthermore, the $\kappa$-geometric (resp.\ $\kappa$-(sub-)first-order) sequents that are valid in $y(U^\kappa_T)$ are precisely those that can be deduced from $T$ in the $\kappa$-geometric (resp.\ $\kappa$-(sub-)first-order) deduction-system.

Similarly, the $\kappa$-first-order sequents valid in $y(U^\kappa_T)$ in $\Sh(\SynB_\kappa(T), J_\kappa)$ are precisely those that can be deduced from the $\kappa$-classical deduction-system.
\end{corollary}
\begin{proof}
The fact that $y(U^\kappa_T)$ is a model of $T$, and hence at least satisfies the sequents deducible from $T$, follows from $y$ being a $\kappa$-geometric (resp.\ $\kappa$-(sub-)Heyting) functor (\thref{yoneda-geometric-sub-heyting}). For the other direction, if a sequent $\phi(x) \vdash_x \psi(x)$ is valid in $y(U^\kappa_T)$ then that means that $y([\phi(x)]) = \{x : \phi(x)\}^{y(U^\kappa_T)} \leq \{x : \psi(x)\}^{y(U^\kappa_T)} = y([\psi(x)])$, and so $[\phi(x)] \leq [\psi(x)]$ because $y$ is full and faithful.
\end{proof}
\begin{theorem}[Intuitionistic completeness theorems]
\thlabel{intuitionistic-completeness-theorems}
Let $T$ be a geometric (resp.\ infinitary (sub-)first-order) theory. If a geometric (resp.\ infinitary (sub-)first-order) sequent $\sigma$ is valid in every model of $T$ in every topos then $\sigma$ deducible from $T$ in the geometric (resp.\ infinitary (sub-)first-order) deduction-system.
\end{theorem}
\begin{proof}
The case for geometric logic is well known and the infinitary first-order case is \cite[Corollary 3.4]{butz_classifying_1998}. With our current setup, the proof for the latter works for all three cases, so we give it for the sub-first-order case.

Let $\sigma$ be a infinitary sub-first-sequent and let $\kappa$ be such that $T$ and $\kappa$ are $\kappa$-sub-first-order. By assumption $\sigma$ is valid in $y(U^\kappa_T)$ in the topos $\Sh(\SynSFO_\kappa(T), J_\kappa)$. We then use \thref{validity-in-yoneda-universal-model} to conclude that $\sigma$ can be deduced from $T$ using the $\kappa$-sub-first-order deduction-system, as required.
\end{proof}
\begin{corollary}
\thlabel{first-order-conservative-over-sub-first-order}
Let $T$ be an infinitary sub-first-order theory. Then any infinitary sub-first-order sequent that is derivable from $T$ in the infinitary first-order deduction system, is already derivable in the infinitary sub-first-order deduction-system.
\end{corollary}
\begin{proof}
Since every topos is sound for the infinitary first-order deduction-system, sequents that can be derived from $T$ in this deduction-system are valid in every model of $T$ in every topos. We can thus apply \thref{intuitionistic-completeness-theorems}.
\end{proof}
\begin{lemma}
\thlabel{sub-first-order-modulo-classical}
Let $\kappa$ be a cardinal. For any $\kappa$-first-order formula $\phi(x)$ there is a $\kappa$-sub-first-order formula $\phi'(x)$ that is equivalent to $\phi(x)$ modulo the $\kappa$-classical deduction-system.
\end{lemma}
\begin{proof}
Note that negation is allowed in sub-first-order logic, as $\neg \psi$ is just $\psi \to \bot$. We can now simply replace the ``forbidden'' connectives in $\phi(x)$ as follows: we replace every occurrence of $\forall y \psi(x, y)$ by $\neg \exists \neg \psi(x, y)$, where $\psi(x, y)$ is a subformula of $\phi(x)$, and we replace infinitary $\bigwedge_{i \in I} \psi(x, y)$ by $\neg \bigvee_{i \in I} \neg \psi_i(x, y)$.
\end{proof}
\begin{lemma}
\thlabel{boolean-topos-with-syntactic-model}
Let $\kappa$ be a cardinal and let $T$ be a $\kappa$-first-order theory. Then there is a Boolean topos $\E$ with a model $M_T$ of $T$ such that the $\kappa$-first-order sequents that are valid in $T$ are precisely those that are $T$-provable in the $\kappa$-classical deduction-system.
\end{lemma}
\begin{proof}
Let $\E = \Sh_{\neg \neg}(\Sh(\SynB_\kappa(T), J_\kappa))$ and write $i: \E \rightleftarrows \Sh(\SynB_\kappa(T), J_\kappa): a$ for the inclusion and sheafification functors. Set $M_T = ay(U^\kappa_T)$, we will verify that $M_T$ is as claimed.

By \thref{sub-first-order-modulo-classical} we have that every $\kappa$-first-order sequent $\sigma$ is equivalent to a $\kappa$-sub-first-order sequent $\sigma'$ modulo the $\kappa$-classical deduction-system. So we find $\kappa$-sub-first-order $T'$ that is equivalent to $T$ modulo the $\kappa$-classical deduction-system, by replacing every sequent $\sigma$ in $T$ by $\sigma'$. Then $U^\kappa_T$ is a model of $T'$ because $\SynB_\kappa(T)$ is $\kappa$-Boolean. Since $y$ and $a$ are both $\kappa$-sub-Heyting (by \thref{yoneda-geometric-sub-heyting} and \cite[Corollary 1.8]{johnstone_open_1980} respectively), we have that $M_T$ is a model of $T'$. Since $\E$ is Boolean we then conclude that $M_T$ is also a model of $T$.

By soundness we then have that the sequents valid in $M_T$ are at least those that are $T$-provable in the $\kappa$-classical deduction-system. Let now $\psi(x) \vdash_x \chi(x)$ be a $\kappa$-first-order sequent that is valid in $M_T$, by \thref{validity-in-yoneda-universal-model} it suffices to show that this sequent is valid in $y(U^\kappa_T)$. Letting $\phi(x)$ be the formula $\psi(x) \to \chi(x)$ and taking $\phi'(x)$ to be its classically equivalent $\kappa$-sub-first-order version from \thref{sub-first-order-modulo-classical}, this further reduces to showing that $\neg \neg \phi'(x)$ holds in $y(U^\kappa_T)$. We then have that
\[
a(\{x : \neg \neg \phi'(x)\}^{y(U^\kappa_T)}) =
\{x : \neg \neg \phi'(x)\}^{M_T} =
\{x : \phi(x)\}^{M_T} =
X^{M_T} =
a(X^{y(U^\kappa_T)}).
\]
Since $\{x : \neg \neg \phi'(x)\}^{y(U^\kappa_T)}$ is a closed subobject of $X^{y(U^\kappa_T)}$ for the $\neg \neg$-topology and the sheafification functor $a$ induces an isomorphism between closed subobjects of $X^{y(U^\kappa_T)}$ and subobjects of $a(X^{y(U^\kappa_T)})$ (see e.g.\ \cite[Corollary V.3.8]{maclane_sheaves_1994}), we conclude that $\{x : \neg \neg \phi'(x)\}^{y(U^\kappa_T)} = X^{y(U^\kappa_T)}$, as required.
\end{proof}
\begin{theorem}[Classical completeness theorem]
\thlabel{classical-completeness-theorem}
Let $T$ be an infinitary first-order theory. If a sequent $\sigma$ is valid in every model of $T$ in every Boolean topos then $\sigma$ deducible from $T$ in the infinitary classical deduction-system.
\end{theorem}
\begin{proof}
Let $\kappa$ be a cardinal such that $\sigma$ and $T$ are $\kappa$-first-order. Let $\E$ and $M_T$ be as in \thref{boolean-topos-with-syntactic-model}. By assumption $\sigma$ is valid in $M_T$ and so we conclude that $\sigma$ is deducible from $T$ in the infinitary classical deduction-system.
\end{proof}
\begin{corollary}
In the completeness results \thref{intuitionistic-completeness-theorems,classical-completeness-theorem} we may restrict to localic toposes (localic Boolean toposes in the latter).
\end{corollary}
\begin{proof}
This improvement is already present for infinitary first-order logic in \cite[Corollary 3.3]{butz_classifying_1998}, the same idea applies to the other logics: it follows because every topos admits an open surjective geometric morphism from a localic topos (see e.g.\ \cite[Theorem IX.9.1]{maclane_sheaves_1994}). For the Boolean case we have by Barr's theorem that every topos admits a surjective geometric morphism from a topos of sheaves on a complete Boolean algebra (see e.g.\ \cite[Theorem IX.9.2]{maclane_sheaves_1994}), so in particular such a topos is Boolean and localic. Since the codomain of this geometric morphism is Boolean the geometric morphism is open by \thref{boolean-iff-every-map-into-is-open}.
\end{proof}
We had restricted our deduction-systems to work with disjunctions and conjunctions of size $< \kappa$ for technical reasons (see the discussion after \thref{soundness}). We can now prove that this is not a genuine restriction on the strength of the deduction-system.
\begin{theorem}
\thlabel{kappa-deduction-is-infinitary}
Let $\kappa$ be a cardinal. Then a $\kappa$-geometric (resp.\ $\kappa$-(sub-)first-order) sequent $\sigma$ is deducible from a $\kappa$-geometric (resp.\ $\kappa$-(sub-)first-order) theory $T$ in the $\kappa$-geometric (resp.\ $\kappa$-(sub-)first-order) deduction-system if and only if it is deducible from $T$ in the geometric (resp.\ infinitary (sub-)first-order) deduction-system.

Similarly, a $\kappa$-first-order sequent is deducible from $T$ in the $\kappa$-classical deduction system if and only if it is deducible from $T$ in the infinitary classical deduction system
\end{theorem}
\begin{proof}
We prove the first-order version of the statement, it should be clear how to adjust the argument for the other two cases. The left to right direction is trivial, we prove the converse. Consider the model $y(U^\kappa_T)$ of $T$ in $\Sh(\SynFO_\kappa(T), J_\kappa)$. Since the topos $\Sh(\SynFO_\kappa(T), J_\kappa)$ is in particular $\infty$-Heyting we have that the infinitary first-order deduction-system is sound for it. Hence $\sigma$ must be valid in $y(U^\kappa_T)$ and so we conclude by \thref{validity-in-yoneda-universal-model} that $\sigma$ can be deduced from the $\kappa$-first-order deduction-system.

For the classical version we use the $\E$ and $M_T$ from \thref{boolean-topos-with-syntactic-model} instead of $\Sh(\SynFO_\kappa(T), J_\kappa)$ and $y(U^\kappa_T)$.
\end{proof}
It is well known that one can construct $\Set[T]$ as $\Sh(\SynG_\kappa(T), J_\kappa)$ for large enough $\kappa$. This is because every geometric formula is equivalent to a disjunction of regular formulas (i.e.\ formulas that only use conjunction and existential quantification). So up to provable equivalence there is only a set of them, and we just pick $\kappa$ bigger than the cardinality of this set (though we could even do with smaller $\kappa$). Classifying toposes for the other fragments of logic may not exist (see \thref{classifying-toposes-counterexamples}). As soon as a theory has a classifying topos for its appropriate logic, we get a strong version of completeness. In fact, it follows from our main theorems that this strong version is equivalent to having a classifying topos for the appropriate logic, because for such theories there will only be a set of formulas in the relevant logic, up to provable equivalence modulo the relevant theory.
\begin{theorem}[Strong completeness]
\thlabel{strong-completeness}
Let $T$ be any theory.
\begin{enumerate}[label=(\roman*)]
\item If $T$ is geometric then the geometric sequents that can be deduced from $T$ in the geometric deduction-system are precisely those that are valid in $G_T$ in the topos $\Set[T]$.
\item If $T$ is infinitary sub-first-order such that $\SetSFO[T]$ exists, then the infinitary sub-first-order sequents that can be deduced from $T$ in the infinitary sub-first-order deduction-system are precisely those that are valid in $G_T$ in the topos $\SetSFO[T]$.
\item If $T$ is infinitary first-order such that $\SetFO[T]$ exists, then the infinitary first-order sequents that can be deduced from $T$ in the infinitary first-order deduction-system are precisely those that are valid in $G_T$ in the topos $\SetFO[T]$.
\item If $T$ is infinitary first-order such that $\SetB[T]$ exists, then the infinitary first-order sequents that can be deduced from $T$ in the infinitary classical deduction-system are precisely those that are valid in $G_T$ in the topos $\SetB[T]$.
\end{enumerate}
\end{theorem}
\begin{proof}
Statement (i) is well-known and (iii) is implicit in \cite{butz_classifying_1998}. We still mention them here because all statements admit the same proof (though overkill for (i)).

Clearly the model $G_T$ satisfies all sequents that are deducible from $T$. Conversely, if $G_T$ satisfies a sequent $\sigma$ then it must be satisfied by every model in every topos, because every such model is the inverse image of $G_T$ under an appropriate geometric morphism. So by completeness \thref{intuitionistic-completeness-theorems} we have that $\sigma$ must be deducible from $T$ in the appropriate deduction-system. For statement (iv) we restrict to Boolean toposes and apply \thref{classical-completeness-theorem} instead.
\end{proof}

\section{The sub-first-order classifying topos}
\label{sec:the-sub-first-order-classifying-topos}
We recall the main result from \cite{butz_classifying_1998} about the existence of $\SetFO[T]$. We also briefly recall its proof, as we base our proof strategy for the existence of $\SetSFO[T]$ on it.
\begin{fact}[{\cite[Theorem 5.3]{butz_classifying_1998}}]
Let $T$ be an infinitary first-order theory $T$. Then $\SetFO[T]$ exists if and only if $T$ is \emph{locally small}, i.e.\ there is only a set of infinitary first-order formulas, up to $T$-provable equivalence in the infinitary first-order deduction-system.
\end{fact}
The left to right direction follows quickly from completeness, see for example also the proof of \thref{sub-first-order-classifying-topos-existence} (i) $\implies$ (ii). For the other direction they construct an extension of $T$, called $\overline{T}$, which is essentially the theory of open geometric morphisms into $\Sh(\SynFO_\kappa(T), J_\kappa)$. Then, using the local smallness condition, one can pick $\kappa$ such that $T$ and $\overline{T}$ are equivalent, and so we can take $\SetFO[T] = \Sh(\SynFO_\kappa(T), J_\kappa)$. We will follow this proof strategy, but restrict ourselves to sub-first-order logic everywhere and construct a theory of sub-open geometric morphisms $T^\sfo$. The characterisation of infinitary sub-first-order theories that admit a sub-first-order classifying topos is then very similar.
\begin{definition}
\thlabel{sub-locally-small}
Let $\kappa$ be a cardinal. A $\kappa$-sub-first-order theory $T$ is called \emph{$\kappa$-sub-locally small} if in any given context there are $< \kappa$ many $\kappa$-sub-first-order formulas, up to $T$-provable equivalence in the infinitary sub-first-order deduction-system.

We call $T$ \emph{sub-locally small} if there is only a set of infinitary sub-first-order formulas, up to $T$-provable equivalence in the infinitary sub-first-order deduction-system.
\end{definition}
We will see in \thref{sub-first-order-classifying-topos-existence} that a theory is sub-locally small if and only if it is $\kappa$-sub-locally small for some $\kappa$.

For the fact below we recall that $\kappa$-geometric functors are precisely the flat and $J_\kappa$-continuous functors (\thref{flat-continuous-is-geometric}).
\begin{fact}[{\cite[Corollary 1.4]{butz_classifying_1998}}]
\thlabel{butz-johnstone-sub-open-condition}
Let $\kappa$ be a cardinal and let $\C$ be a $\kappa$-geometric category and let $\E$ be a topos. Let $F: \C \to \E$ be a $\kappa$-geometric functor and let $f: \E \to \Sh(\C, J_\kappa)$ be the corresponding geometric morphism under the continuous version of Diaconescu's theorem (\thref{continuous-diaconescu}). Then $f$ is sub-open if and only if for each object $A$ in $\C$ and any two $J_\kappa$-closed sieves $S_1$ and $S_2$ on $A$ we have that the image of the induced map
\[
\coprod_{\{g: B \to A \mid g^*(S_1) = g^*(S_2)\}} F(B) \to F(A)
\]
is the subobject $f^*(S_1) \leftrightarrow f^*(S_2)$ of $F(A)$.
\end{fact}
We can use \thref{butz-johnstone-sub-open-condition} to axiomatise sub-open geometric morphisms. For this we will need to make use of implications, so we work in sub-first-order logic.
\begin{definition}
\thlabel{theory-sfo}
Let $\kappa$ be a cardinal and let $T$ be a $\kappa$-sub-first-order theory. We define $T^\sfo$ as the extension of $T$ where we add the following sequents. Let $[\phi(y)]$ be an object in $\SynSFO_\kappa(T)$ and let $S_1 = \{[\alpha_i(x_i, y)]\}_{i \in }$ and $S_2 = \{[\beta_j(x_j, y)]\}_{j \in J}$ be two $J_\kappa$-closed sieves on $[\phi(y)]$. Write $\{[\gamma_k(x_k, y)]\}_{k \in K}$ for the set of all those arrows $g$ into $[\phi(y)]$ such that $g^*(S_1) = g^*(S_2)$. For any such object and two such sieves we then add the double sequent
\[
\bigvee_{k \in K} \exists x_k \gamma(x_k, y) \dashv \vdash_y \bigvee_{i \in I} \exists x_i \alpha_i(x_i, y) \leftrightarrow \bigvee_{j \in J} \exists x_j \beta_j(x_j, y).
\]
\end{definition}
We note that if $T$ is $\kappa$-sub-first-order then $T^\sfo$ may not be $\kappa$-sub-first-order anymore: the sequents we add may contain bigger disjunctions. However, we only add a set of sequents, and all of these sequents are infinitary sub-first-order. So $T^\sfo$ is $\lambda$-sub-first-order for some $\lambda \geq \kappa$.
\begin{proposition}
\thlabel{theory-sfo-models-are-sub-open-geometric-morphisms}
Let $\kappa$ be a cardinal and let $T$ be a $\kappa$-sub-first-order theory. Then for every topos $\E$ there is an equivalence
\[
\SubOpen(\E, \Sh(\SynSFO_\kappa(T), J_\kappa)) \simeq \TMod[T^\sfo](\E)_{\infty \to},
\]
which is given by sending a sub-open geometric morphism $f$ to the model $f^*(y(U^\kappa_T))$.
\end{proposition}
\begin{proof}
Let $f: \E \to \Sh(\SynSFO_\kappa(T), J_\kappa)$ be a sub-open geometric morphism. Then $f^*(y(U^\kappa_T))$ is a model of $T$. The extra sequents in $T^\sfo$ are a direct translation into internal logic of the condition in \thref{butz-johnstone-sub-open-condition}, so they are valid in $f^*(y(U^\kappa_T))$. Hence $f^*(y(U^\kappa_T))$ is a model of $T^\sfo$. Conversely, a model $M$ of $T^\sfo$ in $\E$ is in particular a model of $T$ and so corresponds to a $\kappa$-sub-Heyting functor $F: \SynSFO_\kappa(T) \to \E$ by \thref{functors-as-models}. As $F$ is in particular $\kappa$-geometric and its image, which is $M$, is a model of $T^\sfo$ we have that the geometric morphism $f: \E \to \Sh(\SynSFO_\kappa(T), J_\kappa)$ that corresponds to it under the continuous version of Diaconescu's theorem (\thref{continuous-diaconescu}) is sub-open. Finally, we get infinitary sub-elementary morphisms on the right-hand side because the arrows there correspond to natural transformation between the inverse image parts of sub-open geometric morphisms, which are $\infty$-sub-Heyting functors and so these natural transformations correspond to infinitary sub-elementary morphisms (see the proof of \thref{functors-as-models}).
\end{proof}
\begin{corollary}
\thlabel{theory-sfo-is-full-sub-first-order-theory-of-yoneda-syntactic-model}
Let $\kappa$ be a cardinal and $T$ a $\kappa$-sub-first-order theory. Then:
\begin{enumerate}[label=(\roman*)]
\item $y(U^\kappa_T)$ in $\Sh(\SynSFO_\kappa(T), J_\kappa)$ is a model of $T^\sfo$;
\item any infinitary sub-first-order sequent valid in $y(U^\kappa_T)$ is deducible from $T^\sfo$ in the infinitary sub-first-order deduction-system;
\item any $\kappa$-sub-first-order sequent that is deducible from $T^\sfo$ in the infinitary sub-first-order deduction-system is already deducible from $T$ in that same deduction-system;
\item every infinitary sub-first-order formula is equivalent to a disjunction of $\kappa$-sub-first-order formulas, modulo $T^\sfo$ in the infinitary sub-first-order deduction-system.
\end{enumerate}
\end{corollary}
Note that (i) and (ii) above say that the $T^\sfo$ is essentially the full infinitary sub-first-order theory of $y(U^\kappa_T)$, while (iv) tells us that $T^\sfo$ is sub-locally small.
\begin{proof}
We give a brief argument each item.
\begin{enumerate}[label=(\roman*)]
\item We have that $y(U^\kappa_T)$ corresponds to the identity on $\Sh(\SynSFO_\kappa(T), J_\kappa)$ under the equivalence in \thref{theory-sfo-models-are-sub-open-geometric-morphisms}.
\item This follows from completeness \thref{intuitionistic-completeness-theorems} and \thref{theory-sfo-models-are-sub-open-geometric-morphisms}.
\item Any $\kappa$-sub-first-order sequent that is deducible from $T^\sfo$ must be valid in $y(U^\kappa_T)$ by item (i) and soundness. The result then follows from \thref{validity-in-yoneda-universal-model}.
\item Subobjects in $y(U^\kappa_T)$ correspond to $J_\kappa$-closed sieves in $\SynSFO_\kappa(T)$. As a subobject they are disjunctions of representable subobjects, which in turn correspond to $\kappa$-sub-first-order formulas. Now use item (ii).
\end{enumerate}
\end{proof}
\begin{theorem}
\thlabel{sub-first-order-classifying-topos-existence}
The following are equivalent for an infinitary sub-first-order theory $T$.
\begin{enumerate}[label=(\roman*)]
\item The sub-first-order classifying topos $\SetSFO[T]$ exists.
\item The theory $T$ is sub-locally small.
\item There is $\kappa$ such that $T$ is $\kappa$-sub-locally small.
\end{enumerate}
Furthermore, if $\kappa$ is as in (iii) then we may take $\SetSFO[T] = \Sh(\SynSFO_\kappa(T), J_\kappa)$ with $G_T = y(U^\kappa_T)$.
\end{theorem}
\begin{proof}
Deducability and equivalence of sequents and formulas in this proof is all taken with respect to the infinitary sub-first-order deduction-system, modulo $T$.

\underline{(i) $\implies$ (ii).} The deducible infinitary sub-first-order sequents are precisely those that are valid in $G_T$ in $\SetSFO[T]$ (\thref{strong-completeness}). So two infinitary sub-first-order formulas are equivalent if and only if they are interpreted as the same subobjects in $G_T$. Since there is only a set of subobjects in $G_T$ we conclude that $T$ must be sub-locally small.

\underline{(ii) $\implies$ (iii).} Given a set $S$ of representative formulas, we can take $\kappa$ to be such that $T$ is $\kappa$-sub-first-order, every formula in $S$ is $\kappa$-sub-first-order and $\kappa > |S|$.

\underline{(ii) $\implies$ (i).} Let $\kappa$ be such that $T$ is $\kappa$-sub-locally small. We have that the disjunctions in the extra sequents in $T^\sfo$ are all equivalent to a disjunction of size $< \kappa$. So the extra sequents in $T^\sfo$ are all equivalent to $\kappa$-sub-first-order sequents. By \thref{theory-sfo-is-full-sub-first-order-theory-of-yoneda-syntactic-model}(iii) they are then already deducible from $T$. We thus have that $T$ and $T^\sfo$ are equivalent, and so by \thref{theory-sfo-models-are-sub-open-geometric-morphisms} we have the following equivalence for every topos $\E$:
\[
\SubOpen(\E, \Sh(\SynSFO_\kappa(T), J_\kappa)) \simeq \TMod[T^\sfo](\E)_{\infty \to} = \TMod(\E)_{\infty \to},
\]
which is given by sending a sub-open geometric morphism $f$ to the model $f^*(y(U^\kappa_T))$.
\end{proof}
\begin{corollary}
\thlabel{kappa-sub-locally-small-corollaries}
Let $T$ be a $\kappa$-sub-locally small theory, then any infinitary sub-first-order formula is equivalent to a disjunction of $\kappa$-sub-first-order formulas, modulo $T$ in the infinitary sub-first-order deduction-system.
\end{corollary}
\begin{proof}
This could be deduced from the construction of $\SetSFO[T]$ and strong completeness (see the proof of \thref{classically-kappa-locally-small-corollaries}(ii)), but it also follows from the proof of \thref{sub-first-order-classifying-topos-existence}. There it is shown that $T$ and $T^\sfo$ are equivalent, so we can immediately apply \thref{theory-sfo-is-full-sub-first-order-theory-of-yoneda-syntactic-model}(iv).
\end{proof}
\begin{theorem}
\thlabel{topos-as-sub-first-order-classifying-topos}
Every topos is the sub-first-order classifying topos of some infinitary sub-first-order theory.
\end{theorem}
\begin{proof}
Let $\E$ be any topos and let $T$ be a geometric theory such that $\E \simeq \Set[T]$ (e.g.\ take $T$ to be the theory of geometric morphisms into $\E$, see for example \cite[Theorem 2.1.11]{caramello_theories_2017}). Let $\kappa$ be such that $T$ is $\kappa$-geometric, so we have $\E = \Sh(\SynG_\kappa(T), J_\kappa)$. Define $T'$ by adding sequents like in \thref{theory-sfo}, but we restrict ourselves to objects and closed sieves in the site $(\SynG_\kappa(T), J_\kappa)$. Then by the same proof as \thref{theory-sfo-models-are-sub-open-geometric-morphisms} we have for all toposes $\F$ that
\[
\SubOpen(\F, \E) \simeq \TMod[T'](\F)_{\infty \to}, 
\]
which establishes that $\E$ is the sub-first-order classifying topos of $T'$.
\end{proof}

\section{The Boolean classifying topos}
\label{sec:the-boolean-classifying-topos}
The construction of the Boolean classifying topos is easier than that of the (sub-)first-order classifying topos, because any geometric morphism into a Boolean topos is automatically open and so we do not need to concern ourselves with the theory of (sub-)open geometric morphisms. This also makes the proof very different in nature, as the key point is now to prove that the topos we construct is Boolean. Interestingly, the characterisation remains similar.
\begin{definition}
\thlabel{classically-locally-small}
Let $\kappa$ be a cardinal. A $\kappa$-first-order theory $T$ is called \emph{classically $\kappa$-locally small} if in any given context there are $< \kappa$ many $\kappa$-first-order formulas, up to $T$-provable equivalence in the infinitary classical deduction-system.

We call $T$ \emph{classically locally small} if there is only a set of infinitary first-order formulas, up to $T$-provable equivalence in the infinitary classical deduction-system.
\end{definition}
Just like for sub-locally small, we have that a theory $T$ is classically locally small if and only if it is classically $\kappa$-locally small for some $\kappa$, this time by \thref{boolean-classifying-topos-existence}.
\begin{definition}
\thlabel{syntactic-category-of-consistent-formulas}
Let $\kappa$ be a cardinal. For a $\kappa$-first-order theory $T$ let $\SynConsB_\kappa(T)$ be the full subcategory of $\SynB_\kappa(T)$ consisting of those $[\phi(x)]$ such that $\phi$ is consistent with $T$ modulo classical provability (equivalently: $[\phi(x)]$ is not isomorphic to $[\bot(x)]$). We also denote by $J_\kappa$ the topology on $\SynConsB_\kappa(T)$ that is induced by the topology $J_\kappa$ on $\SynB_\kappa(T)$.
\end{definition}
The above definition is inspired by \cite{blass_boolean_1983} and its point is to show that $J_\kappa = J_{\neg \neg}$ so that the resulting topos is Boolean. However, for $\SynB_\kappa(T)$ this cannot be done, because then $J_\kappa$ has empty covering sieves. Those are exactly the sieves on the objects represented by inconsistent formulas, and by throwing those away we do not change the topos of sheaves anyway.

For the reader's convenience we recall the definition of the $\neg \neg$-topology.
\begin{definition}
\thlabel{double-negation-topology}
The \emph{$\neg \neg$-topology} $J_{\neg \neg}$ on a small category $\C$ is defined as follows. A sieve $S$ on an object $C$ is covering if for any $f: D \to C$ there is $g: E \to D$ such that $fg \in S$.
\end{definition}
\begin{lemma}
\thlabel{j_kappa-subseteq-j_negneg}
Let $\kappa$ be a cardinal and $T$ a $\kappa$-first-order theory. Then we have $J_\kappa \subseteq J_{\neg \neg}$ as topologies on $\SynConsB_\kappa(T)$.
\end{lemma}
\begin{proof}
Deducability and equivalence of formulas in this proof is all taken with respect to the infinitary classical deduction-system, modulo $T$.

Let $S$ be a covering sieve on $[\phi(y)]$ in $J_\kappa$ and let $[\theta(x, y)]: [\psi(x)] \to [\phi(y)]$ be an arrow in $\SynConsB_\kappa(T)$. By definition, there is a $\kappa$-covering family $\{[\sigma_i(x_i, y)]: [\chi_i(x_i)] \to [\phi(y)]\}_{i \in I}$ in $S$.

Suppose, for a contradiction, that for every $i \in I$ the formula $\theta(x, y) \wedge \exists x_i \sigma_i(x_i, y)$ is inconsistent with $T$. Then $\theta(x, y) \wedge \bigvee_{i \in I} \exists x_i \sigma_i(x_i, y)$ is inconsistent with $T$. However, this formula is equivalent to $\theta(x, y)$ because $\{[\sigma_i(x_i, y)]\}_{i \in I}$ is a $\kappa$-covering family. Therefore, $\theta(x, y)$ would be inconsistent with $T$ and cannot be in $\SynConsB_\kappa(T)$.

There must thus be some $i \in I$ such that $\theta(x, y) \wedge \exists x_i \sigma_i(x_i, y)$ is consistent with $T$. Hence, $\theta(x, y) \wedge \sigma_i(x_i, y)$ is consistent with $T$ and the following commuting diagram exists in $\SynConsB_\kappa(T)$:
\[
\begin{tikzcd}
{\theta(x, y) \wedge \sigma_i(x_i, y)} \arrow[rr, "g"] \arrow[d, "f"'] &  & \chi_i(x_i) \arrow[d, "{[\sigma_i(x_i, y)]}"] \\
\psi(x) \arrow[rr, "{[\theta(x,y)]}"']                                 &  & \phi(y)                                      
\end{tikzcd}
\]
where $f$ and $g$ are the (formal) projections of $x$ and $x_i$ respectively. We thus have that $[\theta(x,y)] \circ f = [\sigma_i(x_i, y)] \circ g$ is in $S$. We conclude that $S$ is also a covering sieve for $J_{\neg \neg}$, as required.
\end{proof}
\begin{lemma}
\thlabel{j_negneg-subseteq-j_kappa}
Let $\kappa$ be a cardinal and $T$ a classically $\kappa$-locally small theory. Then we have $J_{\neg \neg} \subseteq J_\kappa$ as topologies on $\SynConsB_\kappa(T)$.
\end{lemma}
\begin{proof}
Deducability and equivalence of formulas in this proof is all taken with respect to the infinitary classical deduction-system, modulo $T$.

We could apply \cite[Corollary 3.13]{caramello_morgan_2009}. However, we give a direct proof as it offers insight in how the assumption on $\kappa$ is used.

Let $S = \{[\theta_i(x_i, y)]\}_{i \in I}$ be a covering sieve on $[\phi(y)]$ in $J_{\neg \neg}$. By our assumption on $\kappa$ there is $I_0 \subseteq I$ such that $|I_0| < \kappa$ and for all $i \in I$ there is $i_0 \in I_0$ such that $\exists x_i \theta_i(x_i, y)$ and $\exists x_{i_0} \theta_{i_0}(x_{i_0}, y)$ are equivalent.

Assume for a contradiction that $\{[\theta_i(x_i,y)]\}_{i \in I_0}$ is not a $\kappa$-covering family. Define $\eta(y)$ to be
\[
\phi(y) \wedge \neg \bigvee_{i \in I_0} \exists x_i \theta_i(x_i, y),
\]
and note that $\eta(y)$ is consistent. Let $m$ denote the (formal) inclusion $[\eta(y)] \hookrightarrow [\phi(y)]$. Since $S$ is covering for $J_{\neg \neg}$ there must be an arrow $f$ into $[\eta(y)]$ such that $mf \in S$. That is, $mf = [\theta_i(x_i,y)]$ for some $i \in I$. We thus have the following commuting diagram
\[
\begin{tikzcd}
{[\chi(x_i)]} \arrow[r] \arrow[d, "f"'] \arrow[rd, "{[\theta_i(x_i, y)]}" description] & {[\exists x_i \theta_i(x_i, y)]} \arrow[d, "n"] \\
{[\eta(y)]} \arrow[r, "m"']                                                            & {[\phi(y)]}                                    
\end{tikzcd}
\]
So, in $\SynB_\kappa(T)$, $[\chi(x_i)]$ admits an arrow into the pullback of $m$ and $n$. However, letting $i_0 \in I_0$ be such that $\exists x_i \theta_i(x_i, y)$ and $\exists x_{i_0} \theta_{i_0}(x_{i_0}, y)$ are equivalent, this pullback is given by
\[
[\eta(y) \wedge \exists x_i \theta_i(x_i, y)] \cong
[\eta(y) \wedge \exists x_{i_0} \theta_{i_0}(x_{i_0}, y)] \cong
[\bot(y)],
\]
which implies that $\chi(x_i)$ is inconsistent. This contradicts $[\chi(x_i)]$, and hence $f$, being in $\SynConsB_\kappa(T)$. We conclude that $\{[\theta_i(x_i,y)]\}_{i \in I_0}$ must be a covering family, and hence $J_{\neg \neg} \subseteq J_\kappa$.
\end{proof}
\begin{theorem}
\thlabel{boolean-classifying-topos-existence}
The following are equivalent for an infinitary first-order theory $T$.
\begin{enumerate}[label=(\roman*)]
\item The Boolean classifying topos $\SetB[T]$ exists.
\item The theory $T$ is classically locally small.
\item There is $\kappa$ such that $T$ is classically $\kappa$-locally small.
\item There is $\kappa$ such that $\Sh(\SynB_\kappa(T), J_\kappa)$ is Boolean.
\end{enumerate}
Furthermore, if $\kappa$ is as in (iii) we may take the same $\kappa$ in (iv), and for $\kappa$ as in (iv) we have that $\SetB[T] = \Sh(\SynB_\kappa(T), J_\kappa)$ with $G_T = y(U^\kappa_T)$.
\end{theorem}
\begin{proof}
Deducability and equivalence of sequents and formulas in this proof is all taken with respect to the infinitary classical deduction-system, modulo $T$.

\underline{(i) $\implies$ (ii) $\implies$ (iii).} Analogous to the similarly numbered implications in \thref{sub-first-order-classifying-topos-existence}.

\underline{(iii) $\implies$ (iv).} Let $\kappa$ be such that $T$ is classically $\kappa$-locally small. Then by \thref{j_kappa-subseteq-j_negneg,j_negneg-subseteq-j_kappa} we have that $J_\kappa = J_{\neg \neg}$ on $\SynConsB_\kappa(T)$. So by the Comparison Lemma (see e.g.\ \cite[Theorem C2.2.3]{johnstone_sketches_2002_2}) we get
\[
\Sh(\SynB_\kappa(T), J_\kappa) \simeq \Sh(\SynConsB_\kappa(T), J_\kappa) = \Sh(\SynConsB_\kappa(T), J_{\neg \neg}),
\]
because we are just leaving out objects with an empty covering sieve. We thus see that $\Sh(\SynB_\kappa(T), J_\kappa)$ is Boolean.

\underline{(iv) $\implies$ (i).} By the continuous version of Diaconescu's theorem (\thref{continuous-diaconescu}) and \thref{functors-as-models} we have the following equivalence for Boolean toposes $\E$:
\[
\Topos(\E, \Sh(\SynB_\kappa(T), J_\kappa)) \simeq \Geom_\kappa(\SynB_\kappa(T), \E) \simeq \TMod(\E).
\]
Furthermore, the equivalence is given by sending a geometric morphism $f$ to the model $f^*(y(U^\kappa_T))$. We conclude by \thref{boolean-classifying-topos-iff-homomorphisms}.
\end{proof}
\begin{corollary}
\thlabel{classically-kappa-locally-small-corollaries}
Let $T$ be a classically $\kappa$-locally small theory, then any infinitary first-order formula is equivalent to a disjunction of $\kappa$-first-order formulas, modulo $T$ in the infinitary classical deduction-system.
\end{corollary}
\begin{proof}
We can repeat the proof of \thref{theory-sfo-is-full-sub-first-order-theory-of-yoneda-syntactic-model}(iv), with $\SynB_\kappa(T)$ in the place of $\SynSFO_\kappa(T)$ and then apply strong completeness \thref{strong-completeness}.
\end{proof}
\begin{theorem}
\thlabel{boolean-topos-as-boolean-classifying-topos}
Every Boolean topos is the Boolean classifying topos of a geometric theory.
\end{theorem}
\begin{proof}
Let $\E$ be any Boolean topos and let $T$ be a geometric theory such that $\E \simeq \Set[T]$ (e.g.\ take $T$ to be the theory of geometric morphisms into $\E$, see for example \cite[Theorem 2.1.11]{caramello_theories_2017}). Then $\E$ is also the Boolean classifying topos of $T$, so $\E = \SetB[T]$, by \thref{boolean-classifying-topos-iff-homomorphisms}.
\end{proof}
The theorem below is close to \cite[Theorem 1]{blass_boolean_1983}, but there are some subtleties. We defer the discussion to \thref{comparison-to-blass-scedrov}.
\begin{theorem}
\thlabel{omega-categorical}
Let $T$ be an $\omega$-first-order theory. Then the following are equivalent:
\begin{enumerate}[label=(\roman*)]
\item in every context there are only finitely many infinitary first-order formulas, up to $T$-provable equivalence in the infinitary classical deduction-system;
\item $T$ is classically $\omega$-locally small;
\item (if the signature is countable) $T$ has finitely many completions, each of which is $\omega$-categorical (i.e.\ there is only one countable\footnote{Here a set is considered countable if it is finite or countably infinite. For $\omega$-categorical theories we normally only consider complete theories with an infinite model, but allowing the finite case does no harm because a complete theory with a finite model only has that one model (up to isomorphism).} model in $\Set$, up to isomorphism);
\item $\SetB[T]$ exists and is coherent.
\end{enumerate}
In particular, if $T$ is as above then every infinitary first-order formula is equivalent to a finitary one, modulo $T$ in the infinitary classical deduction-system.
\end{theorem}
\begin{proof}
\underline{(i) $\implies$ (ii).} Trivial.

\underline{(ii) $\Longleftrightarrow$ (iii).} Is a well-known model-theoretic result, called the Ryll-Nardzewski theorem (see e.g.\ \cite[Theorem 4.4.1]{marker_model_2006}). There is a subtlety that being classically $\omega$-locally small refers to the infinitary classical deduction-system, but since all formulas (and sequents) involved are finitary, \thref{kappa-deduction-is-infinitary} tells us that we may restrict to the $\omega$-classical deduction-system (i.e.\ the usual finitary classical deduction-system).

\underline{(ii) $\implies$ (iv).} By \thref{boolean-classifying-topos-existence} $\SetB[T]$ exists and is given by $\Sh(\SynB_\omega(T), J_\omega)$. Since $\SynB_\omega(T)$ is a coherent category and $J_\omega$ is exactly the coherent coverage topology, we have that $\SetB[T]$ is coherent.

\underline{(iv) $\implies$ (i).} Every subobject lattice in $\SetB[T]$ is a Boolean algebra whose top element is quasi-compact (because $\SetB[T]$ is coherent). Such Boolean algebras are finite. So the result follows from \thref{strong-completeness}.

The final sentence follows from \thref{classically-kappa-locally-small-corollaries}: any infinitary first-order formula is equivalent to a disjunction of finitary first-order formulas, and following the implication (iv) $\implies$ (i) these disjunctions are finite.
\end{proof}
\section{Relating the various classifying toposes}
\label{sec:relating-the-various-classifying-toposes}
Using the different local smallness conditions, conservativity of first-order logic over sub-first-order logic (\thref{first-order-conservative-over-sub-first-order}) and the fact that any infinitary first-order formula is classically equivalent to a sub-first-order one (\thref{sub-first-order-modulo-classical}), we have that the existence of $\SetFO[T]$ implies the existence of $\SetSFO[T]$, which implies the existence of $\SetB[T]$, for infinitary sub-first-order theories $T$. For infinitary first-order theories $T$ we still have that the existence of $\SetFO[T]$ implies the existence of $\SetB[T]$, for the same reasons. However, we can further relate these toposes by specifying how to construct one from another.
\begin{definition}[{\cite[above Proposition 3.6]{johnstone_open_1980}}]
\thlabel{boolean-core}
The \emph{Boolean core} of a topos $\E$ is the largest subobject $U$ of $1$ in $\E$ such that $\E/U$ is contained in $\Sh_{\neg \neg}(\E)$. We write $\B(\E) = \E/U$, where $U$ is the Boolean core of $\E$.
\end{definition}
\begin{fact}[{\cite[Above Proposition 3.6]{johnstone_open_1980}}]
\thlabel{open-sub-open-factoring-through-boolean-core-negneg}
Let $f: \E \to \F$ be a geometric morphism with $\E$ a Boolean topos. Then
\begin{enumerate}[label=(\roman*)]
\item $f$ is sub-open iff it factors through $\Sh_{\neg \neg}(\F) \hookrightarrow \F$;
\item $f$ is open iff it factors through $\B(\F) \hookrightarrow \F$.
\end{enumerate}
\end{fact}
\begin{theorem}
\thlabel{relating-classifying-toposes}
Let $T$ be an infinitary sub-first-order theory.
\begin{enumerate}[label=(\roman*)]
\item If $\SetFO[T]$ exists then $\SetSFO[T] = \SetFO[T]$ (in particular, it exists) and the generic model is the same.
\item If $\SetSFO[T]$ exists then $\SetB[T] = \Sh_{\neg \neg}(\SetSFO[T])$ (in particular, it exists) and the Boolean generic model is the double negation sheafification of $G_T$ in $\SetSFO[T]$.
\end{enumerate}
Furthermore, for any infinitary first-order theory $T$ we have:
\begin{enumerate}[label=(\roman*)]
\setcounter{enumi}{2}
\item If $\SetFO[T]$ exists then $\SetB[T] = \B(\SetFO[T])$ (in particular, it exists) and the Boolean generic model is $a(G_T)$, where $a$ is the inverse image part of the inclusion $\B(\SetFO[T]) \hookrightarrow \SetFO[T]$.
\end{enumerate}
\end{theorem}
\begin{proof}
We prove each item separately.
\begin{enumerate}[label=(\roman*)]
\item Every sub-open geometric morphism into $\SetFO[T]$ is open: for such a sub-open geometric morphism $f$ we have that $f^*(G_T)$ is a model of $T$ and must thus correspond to an open geometric morphism. By an argument similar to \thref{boolean-classifying-topos-iff-homomorphisms} we have that $\infty$-sub-Heyting morphisms between models of $T$ (in any topos) are $\infty$-Heyting. The result follows.
\item Let $i: \Sh_{\neg \neg}(\SetSFO[T]) \hookrightarrow \SetSFO[T]$ denote the inclusion functor and let $a$ be the associated sheafification functor. For Boolean toposes $\E$ then have equivalences
\[
\Topos(\E, \Sh_{\neg \neg}(\SetSFO[T])) \simeq
\SubOpen(\E, \SetSFO[T]) \simeq
\TMod(\E)_{\infty \to},
\]
where the first equivalence follows from \thref{open-sub-open-factoring-through-boolean-core-negneg} and the fact that $i$ is full and faithful. Taking the generic model in $\Sh_{\neg \neg}(\SetSFO[T])$ to be $a(G_T)$, we have by (the argument of) \thref{boolean-classifying-topos-iff-homomorphisms} that this is indeed the Boolean classifying topos.
\item Analogous to item (ii).
\end{enumerate}
\end{proof}
\begin{example}
\thlabel{sub-first-order-classifying-but-not-first-order}
As discussed, we have that the existence of $\SetFO[T]$ implies the existence of $\SetSFO[T]$, which implies the existence of $\SetB[T]$ and ultimately $\Set[T]$ always exists (of course, for theories $T$ in the relevant logic). None of these implications can be reversed, as we establish here.

In \thref{classifying-toposes-counterexamples} parts (i) and (ii) we saw that taking $T$ to be the propositional theory with two propositional variables and no further axioms has no sub-first-order or first-order classifying topos. On the other hand, the free Boolean algebra on two generators is finite and hence complete, so $T$ is classically locally small, and so $\SetB[T]$ exists.

In part (iii) of \thref{classifying-toposes-counterexamples} we saw that if we take a countable infinity of propositional variables instead, the Boolean classifying topos cannot exist. Since the theory has no axioms, it is in particular geometric.

We finish by giving a class of infinitary sub-first-order theories $T$ for which $\SetSFO[T]$ exists, but that do not have a first-order classifying topos. Let $M$ be any monoid that is not a group. By \thref{topos-as-sub-first-order-classifying-topos} there is an infinitary sub-first-order theory $T$ such that $\Set^{M^\op} = \SetSFO[T]$. We claim that $T$ cannot have a first-order classifying topos. The topos $\Set^{M^\op}$ is two-valued and not Boolean \cite[page 274]{maclane_sheaves_1994}, so $\B(\Set^{M^\op})$ is the trivial topos. We thus have that $\Sh_{\neg \neg}(\Set^{M^\op}) \not \simeq \B(\Set^{M^\op})$. If $\SetFO[T]$ were to exist then by \thref{relating-classifying-toposes}(i) we would have $\SetFO[T] = \SetSFO[T] = \Set^{M^\op}$, but by parts (ii) and (iii) of that same theorem we would then have that $\SetB[T]$ is both $\Sh_{\neg \neg}(\Set^{M^\op})$ and $\B(\Set^{M^\op})$. This is a contradiction, so $T$ cannot have a first-order classifying topos.
\end{example}
\begin{theorem}
\thlabel{for-geometric-theories}
Let $T$ be a geometric theory, then the following are equivalent:
\begin{enumerate}[label=(\roman*)]
\item $\Set[T]$ is Boolean;
\item $\Set[T] \simeq \SetSFO[T] \simeq \SetFO[T] \simeq \SetB[T]$ (in particular, they all exist) and their generic models coincide;
\item at least one of $\SetSFO[T]$ or $\SetFO[T]$ exists;
\item there is a cardinal $\kappa$ such that $T$ is classically $\kappa$-locally small and every $\kappa$-first-order formula is equivalent to a $\kappa$-geometric formula modulo $T$ in the infinitary classical deduction-system.
\end{enumerate}
\end{theorem}
\begin{proof}
\underline{(i) $\implies$ (ii).} If $\Set[T]$ is Boolean then every geometric morphism into it is open by \thref{boolean-iff-every-map-into-is-open}. So we have
\[
\Open(\E, \Set[T]) = \SubOpen(\E, \Set[T]) = \Topos(\E, \Set[T]) \simeq \TMod(\E).
\]
By a similar argument to \thref{boolean-classifying-topos-iff-homomorphisms} we have $\TMod(\E) = \TMod(\E)_{\infty \to} = \TMod(\E)_\infty$, which finishes the proof.

\underline{(ii) $\implies$ (iii).} Trivial.

\underline{(iii) $\implies$ (i).} Suppose that $\SetSFO[T]$ exists (the other case is analogous). Denote by $G_T$ and $G_T^\sfo$ the generic models of $T$ in $\Set[T]$ and $\SetSFO[T]$ respectively. By definition we then find geometric morphisms $f: \Set[T] \to \SetSFO[T]$ and $g: \SetSFO[T] \to \Set[T]$ corresponding to $G_T$ and $G_T^\sfo$ respectively. That is, $G_T \cong f^*(G_T^\sfo)$ and $G_T^\sfo \cong g^*(G_T)$. We thus have $G_T \cong f^* g^*(G_T)$ and $G_T^\sfo \cong g^* f^*(G_T^\sfo)$. So $G_T$, as a model of $T$ in $\Set[T]$, corresponds to both $gf: \Set[T] \to \Set[T]$ and the identity $Id: \Set[T] \to \Set[T]$. Thus $gf$ must be naturally isomorphic to the identity. Similarly, we find that $fg$ must be naturally isomorphic to the identity on $\SetSFO[T]$. We conclude that $f$ and $g$ form an equivalence of categories, and so $\Set[T] \simeq \SetSFO[T]$, and furthermore this equivalence sends the generic models to one another.

As geometric morphisms $\E \to \Set[T]$ correspond to models of $T$ in $\E$, which again correspond to sub-open geometric morphisms $\E \to \SetSFO[T] \simeq \Set[T]$, we have that every geometric morphism into $\Set[T]$ is sub-open. So by \thref{boolean-iff-every-map-into-is-open} we have that $\Set[T]$ is Boolean (for the case of $\SetFO[T]$ the argument works the same, but with open geometric morphisms).

\underline{(ii) $\implies$ (iv).} Pick $\kappa$ such that $T$ is $\kappa$-geometric and strictly bigger than the cardinality of all subobject lattices in $G_T$. As $G_T$ is also the generic model in $\SetB[T]$ we have that $T$ is classically $\kappa$-locally small by (the proof of) \thref{boolean-classifying-topos-existence}. Subobjects of $G_T$ in $\Set[T]$ correspond to (interpretations) of geometric formulas (this is well known, but see the proof of \thref{theory-sfo-is-full-sub-first-order-theory-of-yoneda-syntactic-model}(iv) for example), so every infinitary first-order formula is represented by a subobject that also represents a geometric formula. Using again that $G_T$ is also the generic model in $\SetB[T]$, we thus have by \thref{strong-completeness} that every infinitary first-order formula is equivalent to a geometric formula modulo $T$ in the infinitary classical deduction-system. By our choice of $\kappa$ this geometric formula is $\kappa$-geometric.

\underline{(iv) $\implies$ (i).} There is an obvious functor $F: \SynG_\kappa(T) \to \SynB_\kappa(T)$. By Barr's theorem we have that if a geometric sequent can be deduced from a geometric theory in the infinitary classical deduction-system then it can already be deduced in the geometric deduction-system (see e.g.\ \cite[Proposition D3.1.16]{johnstone_sketches_2002_2}). It follows that $F$ is faithful. By our assumption that every $\kappa$-first-order formula is equivalent to a $\kappa$-geometric formula we have that $F$ is also full and essentially surjective (checking that $F$ is full requires Barr's theorem again). We conclude that $F$ is an equivalence of categories, and hence
\[
\Set[T] \simeq \Sh(\SynG_\kappa(T), J_\kappa) \simeq \Sh(\SynB_\kappa(T), J_\kappa) \simeq \SetB[T],
\]
where the final equivalence follows from \thref{boolean-classifying-topos-existence} since $T$ is classically $\kappa$-locally small.
\end{proof}
\begin{remark}
\thlabel{comparison-to-blass-scedrov}
We recall \cite[Theorem 1]{blass_boolean_1983} in our terminology: let $T$ be a coherent theory (i.e.\ $\omega$-geometric), then its classifying topos $\Set[T]$ is Boolean iff $T$ is classically $\omega$-locally small and every finitary first-order formula is equivalent to a coherent formula modulo $T$ in the finitary classical deduction-system.

As mentioned before, the above is close to our \thref{omega-categorical}. There are some subtle differences: \thref{omega-categorical} is about $\omega$-first-order theories and their Boolean classifying topos, whereas the above is about coherent theories and their classifying topos.

However, we can overcome these differences as follows. Firstly, for $T$ an $\omega$-first-order theory we can construct a coherent theory $T'$ that is Morita equivalent (i.e.\ they have equivalent categories of models in every topos) to $T$ by Morleyising (i.e.\ we introduce a relation symbol for every finitary first-order formula and declare it to be equivalent to that formula). It should be clear that $T$ satisfies the equivalent conditions in \thref{omega-categorical} iff $T'$ satisfies these conditions. Secondly, by \thref{for-geometric-theories} we have that $\Set[T']$ is Boolean iff $\Set[T'] \simeq \SetB[T]$.

We thus see that the equivalence of conditions (ii) and (iv) in \thref{omega-categorical} is equivalent to \cite[Theorem 1]{blass_boolean_1983}.

However, \thref{omega-categorical} also tells us that these conditions are equivalent to having only finitely many infinitary first-order formulas (in a fixed context, up to classical provable equivalence modulo the theory), and in fact that every such formula is equivalent to a finitary one. This conclusion is only possible because of the classical completeness theorem \thref{classical-completeness-theorem}.

Finally, we note that \thref{for-geometric-theories} can be seen as an infinitary version of \cite[Theorem 1]{blass_boolean_1983}.
\end{remark}

\bibliographystyle{alpha}
\bibliography{bibfile}


\end{document}